\documentclass[1 [leqno,11pt]{amsart}
\usepackage{amssymb, amsmath}
\allowdisplaybreaks[4]

\let\pa=\partial
\let\f=\frac
\let\wt=\widetilde
\let\wh=\widehat
\let\D=\Delta

\def\lam{\lambda}
\def\ka{\kappa}
\def\ga{\gamma}
\def\v{{\rm v}}

\def\d{\delta}
\def\ff{\frak{f}}
\def\fg{\frak{g}}
\def\bfg{\bar{\frak{g}}}

\def\cB{{\mathcal B}}
\def\cC{{\mathcal C}}

\def\cF{{\mathcal F}}

\def\cS{{\mathcal S}}

\def\eqdef{\buildrel\hbox{\footnotesize def}\over =}
\def\Z{\mathop{\mathbb Z\kern 0pt}\nolimits}
\def\N{\mathop{\mathbb N\kern 0pt}\nolimits}
\def\Q{\mathop{\mathbb Q\kern 0pt}\nolimits}
\def\R{\mathop{\mathbb R\kern 0pt}\nolimits}
\def\Supp{\mathop{\rm Supp}\nolimits\ }

\def\dive{{\mathop{\rm div}\nolimits}\,}
\def\diveh{{\mathop{\rm div}_{\rm h}\nolimits}\,}
\def\curl{{\mathop{\rm curl}\nolimits}\,}
\def\curlh{{\mathop{\rm curl}_{\rm h}\nolimits}\,}
\def\nablah{\nabla_{\rm h}}

\def\uh{u^{\rm h}}

\def\vh{v^{\rm h}}

\def\h{{\rm h}}

\def\vv{{\rm v}}
\def\La{\Lambda}
\def\Deltah{\Delta_{\rm h}}

\def\baru{\bar u}
\def\baruh{\baru^\h}

\def\ah{a^\h}
\def\wtu{\wt u}

\def\na{\nabla}

\def\dhk{\Delta_k^{\rm h}}

\def\dhkp{\Delta_{k'}^{\rm h}}
\def\dvlp{\Delta_{\ell'}^{\rm v}}

\def\dvj{\Delta_j^{\rm v}}

\def\dvl{\Delta_{\ell}^{\rm v}}
\def\LVLH {L^\infty_{\rm v} (L^2_{\rm h})}

\def\vhl{\vh_\lam}
\def\wg{w_\mu}

\def\vhg{v^\h_\mu}
\def\infB{{\wt L^\infty_t(\cB^{0,\frac12})}}
\def\twoB{{\wt L^2_t(\cB^{0,\frac12})}}
\def\twofB{{\wt L^2_{t,f}(\cB^{0,\frac12})}}
\def\twogB{{\wt L^2_{t,\hbar}(\cB^{0,\frac12})}}

\def\twogB{{\wt L^2_{t,\hbar}(\cB^{0,\frac12})}}
\def\Bh{{\cB^{-\f12,\f12}_4}}

\def\infBT{{\wt L^\infty_T(\cB^{0,\frac12})}}
\def\twoBT{{\wt L^2_T(\cB^{0,\frac12})}}
\def\twofBT{{\wt L^2_{T,\ff}(\cB^{0,\frac12})}}

\def\la{\lambda}
\def\e{\varepsilon}
\def\eqdefa{\buildrel\hbox{\footnotesize def}\over =}

\newcommand{\Rmnum}[1]{\uppercase\expandafter{\romannumeral #1} }

\newcommand{\beq}{\begin{equation}}
\newcommand{\eeq}{\end{equation}}
\newcommand{\ben}{\begin{eqnarray}}
\newcommand{\een}{\end{eqnarray}}
\newcommand{\beno}{\begin{eqnarray*}}
\newcommand{\eeno}{\end{eqnarray*}}

 \numberwithin{equation}{section}
\newcommand{\andf}{\quad\hbox{and}\quad}
\newcommand{\with}{\quad\hbox{with}\quad}

\newtheorem{defi}{Definition}[section]
\newtheorem{thm}{Theorem}[section]
\newtheorem{lem}{Lemma}[section]
\newtheorem{rmk}{Remark}[section]
\newtheorem{col}{Corollary}[section]
\newtheorem{prop}{Proposition}[section]

\begin{document}

\title[Global well-posedness of $3$-D anisotropic Navier-Stokes system ]
{Global well-posedness of $3$-D anisotropic Navier-Stokes system with
small unidirectional derivative}

\author[Y. Liu]{Yanlin Liu}
\address[Y. Liu]{Academy of Mathematics $\&$ Systems Science
and Hua Loo-Keng Center for Mathematical Sciences,
Chinese Academy of Sciences, Beijing 100190, CHINA.}
 \email{liuyanlin@amss.ac.cn}

\author[M. Paicu]{Marius Paicu}
\address[M. Paicu]{Laboratoire de Mathmatique, Universit\'e
Paris Sud, B\^{a}timet 425, 91 405 Orsay, FRANCE.}
\email{marius.paicu@math.u-bordeaux.fr}

\author[P. Zhang]{Ping  Zhang }
\address[ P. Zhang]{  Academy of Mathematics $\&$ Systems Science
and and  Hua Loo-Keng Key Laboratory of
Mathematics,  Chinese Academy of
Sciences, Beijing 100190, CHINA, and School of Mathematical Sciences,
University of Chinese Academy of Sciences, Beijing 100049, China.} \email{zp@amss.ac.cn}

\date{\today}

\begin{abstract}  In \cite{LZ4}, the authors proved that as
long as the one-directional derivative of the initial velocity is sufficiently small in some scaling invariant spaces, then the classical Navier-Stokes system  has
a global unique solution. The goal of this paper is to extend this type of result to the 3-D anisotropic Navier-Stokes system $(ANS)$ with only
horizontal dissipation.  More precisely,
 given initial data $u_0=(u_0^\h,u_0^3)\in \cB^{0,\f12},$ $(ANS)$  has a unique global solution provided that $|D_\h|^{-1}\pa_3u_0$
 is sufficiently small in the scaling invariant space $\cB^{0,\f12}.$
 \end{abstract}

\maketitle

\noindent {\sl Keywords:} Anisotropic Navier-Stokes system, Littlewood-Paley theory, well-posedness

\vskip 0.2cm
\noindent {\sl AMS Subject Classification (2000):} 35Q30, 76D03

\setcounter{equation}{0}

\section{Introduction}

In this paper, we investigate the global well-posedness of the following
$3$-D anisotropic Navier-Stokes system:
\begin{equation*}
(ANS)\quad \left\{\begin{array}{l}
\displaystyle \pa_t u +u\cdot\nabla u-\Delta_\h u=-\nabla p, \qquad (t,x)\in\R^+\times\R^3, \\
\displaystyle \dive u = 0, \\
\displaystyle  u|_{t=0}=u_0,
\end{array}\right.
\end{equation*}
where $\D_\h\eqdef \pa_1^2+\pa_2^2,$ $u$ designates the velocity of the fluid and $p$
the scalar pressure function which guarantees the divergence free condition of the velocity field.

Systems of this type appear in geophysical fluid dynamics (see for instance \cite{CDGGbook, Pedlovsky}).
In fact, meteorologists often modelize turbulent diffusion by putting
a viscosity of the form: $-\mu_\h\D_\h-\mu_3\pa_3^2$,
where $\mu_\h$ and $\mu_3$ are empirical constants,
and $\mu_3$ is usually much smaller than $\mu_\h$.
We refer to the book of Pedlovsky \cite{Pedlovsky},
Chap. $4$ for a  complete discussion about this model.

Considering system $(ANS)$ has
 only horizontal dissipation, it is reasonable to use functional spaces
which  distinguish horizontal derivatives from the vertical one,
for instance, the anisotropic Sobolev space defined as follows:

\begin{defi}\label{defanisob}
{\sl For any $(s,s')$ in $\R^2$, the anisotropic Sobolev space
$H^{s,s'}(\R^3)$ denotes the space of homogeneous tempered distribution
$a$ such that
$$\|a\|^2_{H^{s,s'}} \eqdefa \int_{\R^3} |\xi_{\rm
h}|^{2s}|\xi_3|^{2s'} |\wh a (\xi)|^2d\xi <\infty\with \xi_{\rm
h}=(\xi_1,\xi_2).$$
}\end{defi}

Mathematically, Chemin et al. \cite{CDGG}
 first studied the system $(ANS).$ In particular,  Chemin et al. \cite{CDGG} and Iftimie \cite{Iftimie}
 proved that $(ANS)$ is locally well-posed
with initial data in $L^2\cap H^{0,\f12+\varepsilon}$ for some $\varepsilon>0$,
and is globally well-posed if in addition
\begin{equation}\label{smallCDGG}
\|u_0\|_{L^2}^\varepsilon
\|u_0\|_{H^{0,\f12+\varepsilon}}^{1-\varepsilon}\leq c
\end{equation}
for some sufficiently small constant $c$.

Notice that just as the classical Navier-Stokes system
\begin{equation*}
(NS)\quad \left\{\begin{array}{l}
\displaystyle \pa_t u +u\cdot\nabla u-\Delta u=-\nabla p, \qquad (t,x)\in\R^+\times\R^3, \\
\displaystyle \dive u = 0, \\
\displaystyle  u|_{t=0}=u_0,
\end{array}\right.
\end{equation*}
the system $(ANS)$ has the following scaling invariant property:
\begin{equation}\label{NSscaling}
u_\lambda(t,x)\eqdefa\lambda u(\lambda^2 t,\lambda x) \andf u_{0,\lambda}(x)\eqdefa \lambda u_0(\lambda x),
\end{equation}
which means  that if $u$ is a solution of $(ANS)$ with initial data $u_0$ on $[0,T]$, $u_\la$ determined by \eqref{NSscaling}
is also a solution of $(ANS)$ with initial data $u_{0,\lam}$ on $[0,T/\la^2]$.

It is easy  to observe that
the smallness condition \eqref{smallCDGG}  in \cite{CDGG} is scaling invariant under the scaling transformation \eqref{NSscaling},
nevertheless,  the norm of the  space $H^{0,\f12+\varepsilon}$ is not. To work $(ANS)$ with initial data in the critical spaces,
we  first recall the following anisotropic dyadic operators from \cite{BCD}:
\begin{equation}\begin{split}\label{defparaproduct}
&\Delta_k^{\rm h}a\eqdefa\cF^{-1}(\varphi(2^{-k}|\xi_{\rm h}|)\widehat{a}),
 \quad \Delta_\ell^{\rm v}a \eqdefa\cF^{-1}(\varphi(2^{-\ell}|\xi_3|)\widehat{a}),\\
&S^{\rm h}_ka\eqdefa\cF^{-1}(\chi(2^{-k}|\xi_{\rm h}|)\widehat{a}),
\quad\ S^{\rm v}_\ell a \eqdefa\cF^{-1}(\chi(2^{-\ell}|\xi_3|)\widehat{a}),
\end{split}\end{equation}
where  $\xi_{\rm h}=(\xi_1,\xi_2),$ $\cF a$ or $\widehat{a}$  denotes the Fourier transform of $a$,
while $\cF^{-1} a$ designates the inverse Fourier transform of $a$,
$\chi(\tau)$ and $\varphi(\tau)$ are smooth functions such that
\begin{align*}
&\Supp \varphi \subset \Bigl\{\tau \in \R\,: \, \frac34 \leq
|\tau| \leq \frac83 \Bigr\}\quad\mbox{and}\quad \forall
 \tau>0\,,\ \sum_{j\in\Z}\varphi(2^{-j}\tau)=1;\\
& \Supp \chi \subset \Bigl\{\tau \in \R\,: \, |\tau| \leq
\frac43 \Bigr\}\quad\mbox{and}\quad \forall
 \tau\in\R\,,\ \chi(\tau)+ \sum_{j\geq 0}\varphi(2^{-j}\tau)=1.
\end{align*}

\begin{defi}\label{anibesov}
{\sl We define $\cB^{0,\f12}(\R^3)$ to be the set of  homogenous tempered distribution $a$ so that
$$
\|a\|_{\cB^{0,\f12}}\eqdef \sum_{\ell\in\Z}2^{\f{\ell}2}
\|\dvl a\|_{L^2(\R^3)}<\infty.
$$
}\end{defi}

The above space was first introduced by Iftimie in \cite{Iftimie99} to study the global well-posedness of the classical 3-D
Navier-Stokes system with initial data in the anisotropic functional space.
The second author \cite{Pa02} proved the local well-posedness of $(ANS)$ with any
solenoidal vector field $u_0\in\cB^{0,\f12}$ and also the global well-posedness with small initial data in
$\cB^{0,\f12}.$  This result corresponds to  the
 Fujita-Kato's theorem (\cite{fujitakato}) for the classical Navier-Stokes  system. Moreover, the authors \cite{PZ1,Zhang10}
 proved the the global well-posedness  of  $(ANS)$  with initial data $u_0=(u_0^\h,u_0^3)$ satisfying
  \begin{equation}\label{smallPZ1a}
\|\uh_0\|_{\cB^{0,\f12}}\exp\bigl(C\|u^3_0\|_{\cB^{0,\f12}}^4\bigr)\leq c_0
\end{equation}
for some $c_0$ sufficiently small.

Although the norm of $B^{0,\f12}$ is scaling invariant under the the scaling transformation \eqref{NSscaling}, yet we observe that
the solenoidal vector field \beq
\label{exam} u_0^\e(x) = \sin \Bigl(\frac
{x_{1}}\e\Bigr)\left(0, -
\partial_{3}\varphi, \partial_{2}\varphi\right)
\eeq
is not small in the space $B^{0,\f12}$ no matter how small $\e$ is. In order to find a space so that the norm  of $u_0^\e(x)$ given by \eqref{exam} is small in this space for small $\e,$
Chemin and the third author \cite{CZ07} introduced the following Besov-Soblev type space with negative index:

\begin{defi}\label{animinus}
{\sl
We define the space $\cB^{-\f12,\f12}_4$ to be the set of  a homogenous tempered distribution $a$ so that
$$\|a\|_{\cB^{-\f12,\f12}_4}\eqdef\sum_{\ell\in\Z}
2^{\f\ell2}\Bigl(\bigl(\sum_{k=\ell-1}^\infty 2^{-k}
\|\D^\h_k \D^\vv_\ell a\|_{L^4_\h(L^2_\vv)}^2\bigr)^{\f12}
+\|S^\h_{\ell-1}\D^\vv_\ell a\|_{L^2}\Bigr)<\infty. $$
}
\end{defi}

Chemin and the third author \cite{CZ07} proved the global well-posedness of $(ANS)$
with initial data being small in the space $\cB^{-\f12,\f12}_4.$ In particular, this result ensures the global
well-posedness of $(ANS)$ with initial data $u_0^\e(x)$ given by \eqref{exam} as long as $\e$ is sufficiently small.  Furthermore
 the second and third authors \cite{PZ1} proved the  global well-posedness of $(ANS)$ provided that
 the initial data $u_0=(u_0^\h,u_0^3)$ satisfies
 \begin{equation}\label{smallPZ1}
\|\uh_0\|_{\cB^{-\f12,\f12}_4}\exp\bigl(C\|u^3_0\|_{\cB^{-\f12,\f12}_4}^4\bigr)\leq c_0
\end{equation}
for some $c_0$ sufficiently small. We remark that this result corresponds to  Cannone,  Meyer and  Planchon's result
in \cite{cannonemeyerplanchon} for the classical Nvaier-Stokes system, where the authors proved that
 if the initial data satisfies
 \[
 \|u_{0}\|_{ \dot B^{-1+\frac 3 p}_{p,\infty} }\leq c\nu
 \]
for  some~$p$ greater than~$3$ and some constant~$c$ small enough, then $(NS)$ is globally well-posed.
The end-point result in this direction is due to   Koch and  Tataru \cite{kochtataru} for initial data in the space of
~$\partial BMO.$

On the other hand, motivated by the study of the global well-posedness of the classical Navier-Stokes system
with slowly varying initial data \cite{CG10,CGZ, CZ15},
 the first and third authors
 proved the following theorem for $(NS)$ in \cite{LZ4}:

\begin{thm}\label{thmLZ}
{\sl Let $u_0=(u^h_0,u_0^3)\in H^{-\delta,0}\cap H^2(\R^3)$
for some $\delta\in]0,1[$
with $\pa_3u_0^\h\in H^{-\delta,0}\cap H^{-\f12,0}$.
There exists a universal small positive constant $\e_0$ such that if
\begin{equation}\label{smallLZ}
\|\pa_3u_0\|_{H^{-\f12,0}}^2\exp\bigl(CA_\d(\uh_0)+CB_\d(u_0)\bigr)
\leq\e_0,
\end{equation}
where
\beno
\begin{split}
B_\d(u_0)\eqdefa &\|u_0\|_{H^{-\d,0}}^{\f12}\|u_0\|_{H^{\d,0}}^{\f12}
\|\pa_3u_0\|_{H^{-\d,0}}^{\f12}
\|\pa_3u_0\|_{H^{\d,0}}^{\f12}
\exp\bigl(CA_\d(\uh_0)\bigr) \andf \\
A_\d(u_0^\h)\eqdefa& \biggl(\frac{\|\nabla _{\rm h}\uh_0\|_{\LVLH}^2
\|\uh_0\|_{L^\infty_{\rm v} (B^{-\d}_{2,\infty})_\h}
^{\frac 2 \delta}} {\|\uh_0\|_{\LVLH}^{\frac 2 \d}}
+\|\uh_0\|_{\LVLH}^{2}\Bigr)\cdot
\exp\Big(C_\d(1+\|\uh_0\|_{\LVLH}^4)\biggr)\end{split}
\eeno
then $(NS)$ has a unique global solution
$u\in C\bigl(\R^+;H ^{\frac12}\bigr)\cap
L^2\bigl(\R^+;H ^{\frac 32}\bigr)$.
}\end{thm}

We remark that Theorem \ref{thmLZ} ensures the global well-posedness of $(NS)$  with initial data
\begin{equation}\label{slowinitialdata}
u_0^\varepsilon(x)=(\vh_0 +\varepsilon w^\h_0,w^3_0)(x_\h,\varepsilon x_3) \with \diveh\vh_0=0=\dive w_0
\end{equation}
for $\e\leq\e_0,$ which was first proved in \cite{CG10}. We mention that the proof of Theorem \ref{thmLZ}
requires a regularity criteria in \cite{CZ5}, which can only be proved for the classical Navier-Stokes system so far.

Motivated by \cite{LZ4} and \cite{PZ1, Zhang10}, here we are going to study the global well-posedness of $(ANS)$ with initial data $u_0$ satisfying $\pa_3u_0$ being sufficiently small in some critical spaces.

The main result of this paper states as follows:

\begin{thm}\label{thmmain}
{\sl Let $\La_\h^{-1}$ be a Fourier multiplier with symbol $|\xi_\h|^{-1},$ let $u_0\in \cB^{0,\f12}$ be a solenoidal vector field with $\La_\h^{-1}\pa_3 u_0\in {\cB^{0,\f12}}.$
Then there exist some sufficiently small positive constant $\e_0$ and some universal positive constants $L, M, N$
 so that for $\frak{A}_N\bigl(\|u^\h_0\|_{\cB^{0,\f12}}\bigr)$ given by \eqref{ANuh} if
\begin{equation}\begin{split}\label{smallnesscondition}
\|\Lambda^{-1}_\h\pa_3 u_0\|_{\cB^{0,\f12}}
\exp\Bigl(L\bigl(1+\|u_0^3\|_\Bh^4\bigr)\exp\bigl(M\frak{A}_N^4\bigl(\|u^\h_0\|_{\cB^{0,\f12}}\bigr)\bigr)\Bigr)
\leq \e_0,
\end{split}\end{equation}
 $(ANS)$ has a unique global solution $u=\frak{v}+e^{t\D_h} \begin{pmatrix}
0\\
u^3_{0,{\rm hh}}
\end{pmatrix}$ with
$\frak{v}\in C([0,\infty[\,;\cB^{0,\f12})$ and $\nablah \frak{v}\in L^2([0,\infty[\,;\cB^{0,\f12}),$
where $u_{0,{\rm hh}}^3\eqdefa \sum_{k\geq \ell-1}\D_k^\h\D_\ell^\v u_0^3.$
}\end{thm}

We remark that all the norms of $u^0$ in \eqref{smallnesscondition} is scaling invariant under the scaling transformation \eqref{NSscaling}.
With regular initial data, we may write explicitly the constant $\frak{A}_N\bigl(\|u^\h_0\|_{\cB^{0,\f12}}\bigr).$  For instance,

\begin{col}\label{S1col1}
{\sl Let $u_0\in L^2$ be a solenoidal vector field with $\pa_3u_0\in L^2$ and $\La_\h^{-1}\pa_3 u_0\in {\cB^{0,\f12}}.$
Then there exist some sufficiently small positive constant $\e_0$ and some universal positive constants $L, M$
 so that if \beq \label{smallnessconditiona} \|\Lambda^{-1}_\h\pa_3 u_0\|_{\cB^{0,\f12}}
\exp\Bigl(L\bigl(1+\|u_0^3\|_\Bh^4\bigr)
\exp\left(\exp\bigl(M\|\uh_0\|_{L^2}\|\pa_3\uh_0\|_{L^2}\bigr)\right)\Bigr)
\leq \e_0,\eeq $(ANS)$ has a unique global solution $u$ as in Theorem \ref{thmmain}.
}
\end{col}

\begin{rmk} We have several remarks in order as follows about Theorem \ref{thmmain}:

\begin{itemize}
\item[(a)] It follows from \cite{CZ07} that  $$ \|u_0^3\|_\Bh\lesssim \|u_0^3\|_{\cB^{0,\f12}},$$
so that the smallness condition \eqref{smallnesscondition} and \eqref{smallnessconditiona} can also be formulated as
\begin{equation}\begin{split}\label{smallnessconditionam}
\|\Lambda^{-1}_\h\pa_3 u_0\|_{\cB^{0,\f12}}
\exp\Bigl(L\bigl(1+\|u_0^3\|_{\cB^{0,\f12}}^4\bigr)\exp\bigl(M\frak{A}_N^4\bigl(\|u^\h_0\|_{\cB^{0,\f12}}\bigr)\bigr)\Bigr)
\leq \e_0,
\end{split}\end{equation}
and
\beq \label{smallnessconditionaq} \|\Lambda^{-1}_\h\pa_3 u_0\|_{\cB^{0,\f12}}
\exp\Bigl(L\bigl(1+\|u_0^3\|_{\cB^{0,\f12}}^4\bigr)
\exp\left(\exp\bigl(M\|\uh_0\|_{L^2}\|\pa_3\uh_0\|_{L^2}\bigr)\right)\Bigr)
\leq \e_0.\eeq

\item[(b)] Due to $\dive u_0=0,$ we find
$$\|\Lambda^{-1}_\h\pa_3 u_0\|_{\cB^{0,\f12}}=
\|(\Lambda^{-1}_\h\pa_3\uh_0,-\Lambda^{-1}_\h\diveh \uh_0)\|_{\cB^{0,\f12}}.$$
Therefore the smallness condition \eqref{smallnesscondition} is of a similar type as \eqref{smallPZ1a}.
Yet Roughly speaking, \eqref{smallnesscondition}  requires only $\pa_3 \uh_0$
and $\diveh \uh_0$ to be small in some scaling invariant space, but without any restriction on $\curlh\uh_0$.
Thus the smallness condition  \eqref{smallnesscondition} is weaker than  \eqref{smallPZ1a}.

\item[(c)] Let $w_0$ be a smooth solenoidal vector field, we observe  that the following data
$$u_0^\varepsilon(x)=\bigl(\varepsilon(-\ln\varepsilon)^\delta w^\h_0,
(-\ln\varepsilon)^\delta w^3_0\bigr)(x_\h,\varepsilon x_3) \with  \delta\in ]0,1/4[$$
satisfies \eqref{smallPZ1a} for $\e$ sufficiently small.

While since our smallness condition \eqref{smallnessconditionaq} does not have any restriction on $\curl u_0^\h,$ for any
smooth vector field $v^\h_0$ satisfying $\diveh v^h_0=0,$ we find
\beq \label{S1eq11} u_0^\varepsilon(x)=\bigl(\vh_0+\varepsilon(-\ln\varepsilon)^\delta w^\h_0,
(-\ln\varepsilon)^\delta w^3_0\bigr)(x_\h,\varepsilon x_3)  \with  \delta\in ]0,1/4[ \eeq
satisfy \eqref{smallnessconditionaq} for any $\e$ sufficiently small. Therefore Theorem \ref{thmmain} ensures the global
well-posedness of $(ANS)$ with initial data given by \eqref{S1eq11}. Compared with \eqref{slowinitialdata}, which corresponds to
$\delta=1$ in \eqref{S1eq11}, this type of result is new even for the classical Navier-Stokes system.

\item[(d)]  Given $\phi\in \cS(\R^3)$, we deduce from  Proposition 1.1 in \cite{CZ07} that
$$
\|e^{ix_1/\varepsilon}\phi(x)\|_{\Bh}
\leq C\varepsilon^{\f12}.$$
As a result, we find that for any $\d\in ]0,1/4[$, the following  class of initial data
\beq \label{S1eq12} u_0^\varepsilon(x)=(v^\h, 0)(x_\h,\varepsilon x_3)
+(-\ln\varepsilon)^\delta\sin(x_1/\varepsilon)
\bigl(0,-\varepsilon^{\f12}\pa_3\phi(x_\h,\varepsilon x_3),
\varepsilon^{-\f12}\pa_2\phi(x_\h,\varepsilon x_3)\bigr) \eeq
satisfies the smallness condition \eqref{smallnessconditionam} for small enough
$\e$,
and hence the data given by \eqref{S1eq12} can also generate unique global solution of $(ANS)$.

\item[(e)] Since  all the results that work for the anisotropic Navier-Stokes system $(ANS)$ should automatically do for the
classical Navier-Stokes system $(NS),$
Theorem \ref{thmmain} holds   also for $(NS)$.
\end{itemize}\end{rmk}

Let us end this section with some notations that will be used
throughout this paper.

\noindent{\bf Notations:}   Let $A, B$ be two operators, we denote
$[A;B]=AB-BA,$ the commutator between $A$ and $B,$ for $a\lesssim b$,
we means that there is a uniform constant $C,$ which may be different in each occurrence, such that $a\leq Cb$. We shall denote by~$(a|b)_{L^2}$
the $L^2(\R^3)$ inner product of $a$ and $b.$
 $\left(d_j\right)_{j\in\Z}$ designates a  generic elements on the unit
sphere of $\ell^1(\Z)$, i.e. $\sum_{j\in\Z}d_j=1$.
 Finally, we
denote $L^r_T(L^p_\h(L^q_\v))$ the space $L^r([0,T];
L^p(\R_{x_1}\times\R_{x_2}; L^q(\R_{x_3}))),$ and $\na_\h\eqdefa (\pa_{x_1},\pa_{x_2}),$ $\dive_\h=
\pa_{x_1}+\pa_{x_2}$.

\medskip

\setcounter{equation}{0}
\section{Littlewood-Paley Theory}\label{secLP}

In this section, we shall collect some basic facts on anisotropic
Littlewood-Paley theory.
We first recall the following anisotropic Bernstein inequalities from
\cite{CZ07, Pa02}:
\begin{lem}\label{lemBern}
{\sl Let ${\bf B}_{\h}$ (resp.~${\bf B}_{\rm v}$) a ball
of~$\R^2_{\h}$ (resp.~$\R_{\rm v}$), and~$\cC_{\h}$ (resp.~$\cC_{\rm v}$) a
ring of~$\R^2_{\h}$ (resp.~$\R_{\rm
v}$); let~$1\leq p_2\leq p_1\leq
\infty$ and ~$1\leq q_2\leq q_1\leq \infty.$ Then there holds
\beno
\begin{split}
\mbox{if}\ \ \Supp \wh a\subset 2^k{\bf B}_{\h}&\Rightarrow
\|\partial_{x_{\rm h}}^\alpha a\|_{L^{p_1}_{\rm h}(L^{q_1}_{\rm v})}
\lesssim 2^{k\left(|\alpha|+\f2{p_2}-\f2{p_1}\right)}
\|a\|_{L^{p_2}_{\rm h}(L^{q_1}_{\rm v})};\\
\mbox{if}\ \ \Supp\wh a\subset 2^\ell{\bf B}_{\rm v}&\Rightarrow
\|\partial_{x_3}^\beta a\|_{L^{p_1}_{\rm h}(L^{q_1}_{\rm v})}
\lesssim 2^{\ell\left(\beta+\f1{q_2}-\f1{q_1}\right)} \|
a\|_{L^{p_1}_{\rm h}(L^{q_2}_{\rm v})};\\
\mbox{if}\ \ \Supp\wh a\subset 2^k\cC_{\h}&\Rightarrow
\|a\|_{L^{p_1}_{\rm h}(L^{q_1}_{\rm v})} \lesssim
2^{-kN}\sup_{|\alpha|=N}
\|\partial_{x_{\rm h}}^\alpha a\|_{L^{p_1}_{\rm
h}(L^{q_1}_{\rm v})};\\
\mbox{if}\ \ \Supp\wh a\subset2^\ell\cC_{\rm v}&\Rightarrow
\|a\|_{L^{p_1}_{\rm h}(L^{q_1}_{\rm v})} \lesssim 2^{-\ell N}
\|\partial_{x_3}^N a\|_{L^{p_1}_{\rm h}(L^{q_1}_{\rm v})}.\end{split}\eeno
}
\end{lem}

\begin{defi}\label{B(T)}
{\sl  For any $p\in[1,\infty],$, let us define the
Chemin-Lerner type norm
$$\|a\|_{\wt L^p_T(\cB^{0,\f12})}\eqdef\sum_{\ell\in\Z}2^{\f{\ell}2}
\|\dvl a\|_{L^p_T(L^2(\R^3))}.$$
In particular, we denote
$$\|a\|_{\cB^{0,\f12}(T)}\eqdef\|a\|_{\wt L^\infty_T(\cB^{0,\f12})}
+\|\nablah a\|_{\wt L^2_T(\cB^{0,\f12})}.$$}
\end{defi}

We remark that the inhomogeneous  version of the  anisotropic Sobolev space $H^{0,1}$ can be continuously imbedded into $\cB^{0,\f12}.$ Indeed
for any integer $N$, we deduce from  Lemma \ref{lemBern} that
\beno
\begin{split}
\|a\|_{\cB^{0,\f12}}=&\sum_{\ell\leq N}2^{\f{\ell}2}\|\dvl a\|_{L^2}
+\sum_{\ell>N}2^{\f{\ell}2}\|\dvl a\|_{L^2}\\
\leq &\sum_{\ell\leq N} 2^{\f{\ell}2}\|\dvl a\|_{L^2}
+\sum_{\ell>N}2^{-\f{\ell}2}\|\pa_3\dvl a\|_{L^2}\\
\lesssim & 2^{\f N2}\|a\|_{L^2}+2^{-\f N2}\|\pa_3 a\|_{L^2}.
\end{split}
\eeno
Taking the integer $N$ so that
$2^{N}\sim \|\pa_3 a\|_{L^2}\|a\|_{L^2}^{-1}$ in the above inequality
leads to
\begin{equation}\label{interpolationcB}
\|a\|_{\cB^{0,\f12}}\lesssim\|a\|_{L^2}^{\f12}
\|\pa_3 a\|_{L^2}^{\f12}.
\end{equation}
Along the same line, we have
\begin{equation}\label{interLpcB}
\|a\|_{\wt L^p_T(\cB^{0,\f12})}\lesssim\|a\|_{L^p_T(L^2)}^{\f12}
\|\pa_3 a\|_{L^p_T(L^2)}^{\f12}\quad\forall\ p \in[1,\infty].
\end{equation}

To overcome the difficulty that one can not use Gronwall's inequality in the Chemin-Lerner type norms, we recall the following
time-weighted Chemin-Lerner norm from \cite{PZ1}:
\begin{defi}\label{defpz}
{\sl Let $f(t)\in L^1_{\rm{loc}}(\R_+)$, $f(t)\geq 0$. We define
$$\|a\|_{\widetilde L^2_{T,f}(\cB^{0,\f12})}\eqdefa \sum_{\ell\in\Z} 2^{\f\ell2}
\Bigl(\int_0^Tf(t)\|\dvl a(t)\|_{L^2}^2\,dt\Bigr)^{\frac 12}.$$}
\end{defi}

In order to take into account functions with oscillations in the horizontal variables,
we recall the following  anisotropic Besov type space with negative indices
 from \cite{CZ07}:
\begin{defi}\label{defBh}
{\sl
For any $ p\in[1,\infty],$ we define $$\|a\|_{\wt L^p(\cB^{-\f12,\f12}_4)}\eqdef\sum_{\ell\in\Z}
2^{\f\ell2}\biggl(\Bigl(\sum_{k=\ell-1}^\infty 2^{-k}
\|\D^\h_k \D^\vv_\ell a\|_{L^p_T(L^4_\h(L^2_\vv))}^2\Bigr)^{\f12}
+\|S^\h_{\ell-1}\D^\vv_\ell a\|_{L^p_T(L^2)}\biggr).$$
In particular, we denote
$$\|a\|_{\cB^{-\f12,\f12}_4(T)}\eqdef\|a\|_{\wt L^\infty_T(\Bh)}
+\|\nablah a\|_{\wt L^2_T(\Bh)}.$$
}\end{defi}

%It is easy to observe from Lemma \ref{lemBern} that
%\beq\label{cBBh}
%\|a\|_{\cB^{-\f12,\f12}_4(T)}\lesssim\|a\|_{\cB^{0,\f12}(T)}\quad\mbox{for any }\ \ T>0.
%\eeq

In the sequel, for $a\in \cB^{-\f12,\f12}_4,$ we shall frequently use the following decomposition:
\beq \label{decom}
a=a_{\rm lh}+a_{\rm hh} \with a_{\rm lh}\eqdefa \sum_{\ell\in\Z}S_{\ell-1}^\h\D_\ell^\v a \andf a_{\rm hh}\eqdefa \sum_{k\geq \ell-1}\D_k^\h\D_\ell^\v a.
\eeq

\begin{lem}[Lemma $2.5$ in \cite{CZ07}]\label{lemCZ07}
{\sl For any $a\in\Bh,$ there holds
$$\|e^{t\D_\h}a_{\h\h}\|_{\Bh(\infty)}\lesssim\|a\|_\Bh.$$
}\end{lem}

\begin{defi}\label{Def2.4}
{\sl Let us define
$$\|a\|_{\cB_4^{0,\f12}}\eqdef \sum_{\ell\in\Z}2^{\f{\ell}2}
\|\dvl a\|_{L^4_\h(L^2_\v)}\quad\mbox{and}\quad
\|a\|_{\wt L^4_t(\cB_4^{0,\f12})}\eqdef \sum_{\ell\in\Z}2^{\f{\ell}2}
\|\dvl a\|_{L^4_t(L^4_\h(L^2_\v))}.$$
}\end{defi}

In view of the 2-D interpolation inequality that \beq \label{S2eq8} \|a\|_{L^4(\R^2)}\lesssim \|a\|_{L^2(\R^2)}^{\f12}\|\na_\h a\|_{L^2(\R^2)}^{\f12},\eeq
 we find
 \beq \label{S2eq5}
\begin{split}
\|a\|_{\cB_4^{0,\f12}}^2&\lesssim \Bigl(\sum_{\ell\in\Z}2^{\f{\ell}2}
\|\dvl a\|_{L^2}^{\f12}\|\dvl\nablah a\|_{L^2}^{\f12}\Bigr)^2\\
&\leq \Bigl(\sum_{\ell\in\Z}2^{\f{\ell}2}
\|\dvl a\|_{L^2}\Bigr)\Bigl(\sum_{\ell\in\Z}2^{\f{\ell}2}\|\dvl\nablah a\|_{L^2}\Bigr)
=\|a\|_{\cB^{0,\f12}}\|\nablah a\|_{\cB^{0,\f12}}.
\end{split} \eeq
Similarly, we have
\beq \label{S2eq5a} \|a\|_{\wt L^4_t(\cB_4^{0,\f12})}^2
\lesssim\|a\|_{\wt{L}^\infty_t(\cB^{0,\f12})}\|\nablah a\|_\twoB. \eeq

Before preceding, let us recall Bony's decomposition for the
vertical variable from \cite{Bo81}:
\begin{equation}\label{bony}\begin{split}
 ab=T^\v_a b+ R^\v(a,b)\quad\mbox{with}\quad
T^\v_a b=\sum_{\ell\in\Z}S^\v_{\ell-1}a\Delta^\v_\ell b,\quad
R^\v(a,b)=\sum_{\ell\in\Z}\Delta^\v_\ell a S^\v_{j+2}b.
\end{split}\end{equation}
Sometimes we shall also use  Bony's decomposition for the horizontal variables.

Let us now apply the above basic facts on Littlewood-Paley theory to prove the following proposition:

\begin{prop}\label{lemB4}
{\sl For any $a\in\Bh(T),$ there holds
\beq \label{S2eq7} \|a\|_{\wt L_T^4(\cB_4^{0,\f12})}\lesssim\|a\|_{\Bh(T)}. \eeq
}\end{prop}

\begin{proof} In view of \eqref{decom} and Definition \ref{defBh}, we get, by applying  \eqref{S2eq5a}, that
\beno
\begin{split}
\|a_{\rm lh}\|_{\wt L_T^4(\cB_4^{0,\f12})}\lesssim & \|a_{\rm lh}\|_{\wt L_T^\infty(\cB^{0,\f12})}^{\f12} \|\na_\h a_{\rm lh}\|_{\wt L_T^2(\cB^{0,\f12})}^{\f12}\\
\lesssim &\|a\|_{\wt{L}^\infty_T(\Bh)}^{\f12}
\|\nablah a\|_{\wt{L}^2_T(\Bh)}^{\f12}.
\end{split}
\eeno
Then it remains to prove \eqref{S2eq7} for $a_{\rm hh}.$
Indeed in view of Definition \ref{Def2.4},  we write
\begin{equation*}\label{ineqB41}
\|a_{\rm hh}\|_{\wt L_T^4(\cB_4^{0,\f12})}
=\sum_{\ell\in\Z}2^{\f{\ell}2}
\|(\dvl a_{\rm hh})^2\|_{L_T^2(L^2_\h(L^1_\vv))}^{\f12}.
\end{equation*}
Applying Bony's decomposition for the horizontal variables yields
\begin{equation}\label{ineqB42}
(\dvl a_{\rm hh})^2=\sum_{k\in\Z}S^\h_{k-1}\dvl a_{\rm hh}\dhk\dvl a_{\rm hh}
+\sum_{k\in\Z}S^\h_{k+2}\dvl a_{\rm hh}\dhk\dvl a_{\rm hh}.
\end{equation}
We observe that
\begin{equation*}\begin{split}\label{ineqB43}
\sum_{\ell\in\Z}&2^{\f{\ell}2}\Bigl(\sum_{k\in\Z}
\|S^\h_{k-1}\dvl a_{\rm hh}\dhk\dvl a_{\rm hh}\|_{L_T^2(L^2_\h(L^1_\vv))}\Bigr)^{\f12}\\
&\leq\biggl(\sum_{\ell\in\Z}2^{\f{\ell}2}\Bigl(\sum_{k\in\Z}
2^{-k}\|S^\h_{k-1}\dvl a_{\rm hh}\|_{L_T^\infty(L^4_\h(L^2_\vv))}^2\Bigr)^{\f12}
\biggr)^{\f12}\\
&\qquad\quad\times \biggl(\sum_{\ell\in\Z}2^{\f{\ell}2}\Bigl(\sum_{k\in\Z}
2^{k}\|\dhk\dvl a_{\rm hh}\|_{L_T^2(L^4_\h(L^2_\vv))}^2\Bigr)^{\f12}\biggr)^{\f12}\\
&\leq\biggl(\sum_{\ell\in\Z}2^{\f{\ell}2}\Bigl(\sum_{k\in\Z}
2^{-k}\|S^\h_{k-1}\dvl a_{\rm hh}\|_{L_T^\infty(L^4_\h(L^2_\vv))}^2\Bigr)^{\f12}
\biggr)^{\f12}\|\nablah a_{\rm hh}\|_{\wt{L}^2_T(\Bh)}^{\f12}.
\end{split}\end{equation*}
Whereas we get,  by using Young's inequality, that
\begin{align*}
\sum_{k\in\Z}
2^{-k}\|S^\h_{k-1}\dvl a_{\rm hh}\|_{L_T^\infty(L^4_\h(L^2_\vv))}^2
&=\sum_{k\in\Z}
\Bigl(\sum_{k'\leq k-2}2^{-\f{k-k'}2}2^{-\f{k'}2}\|\dhkp\dvl a_{\rm hh}\|
_{L_T^\infty(L^4_\h(L^2_\vv))}\Bigr)^2\\
&\leq\sum_{k\in\Z}2^{-k}\|\dhk\dvl a_{\rm hh}\|
_{L_T^\infty(L^4_\h(L^2_\vv))}^2.
\end{align*}
As a result, it comes out
$$\sum_{\ell\in\Z}2^{\f{\ell}2}\Bigl(\sum_{k\in\Z}
2^{-k}\|S^\h_{k-1}\dvl a_{\rm hh}\|_{L_T^\infty(L^4_\h(L^2_\vv))}^2\Bigr)^{\f12}
\leq \|a\|_{\wt{L}^\infty_T(\Bh)},$$
and
\beno
\sum_{\ell\in\Z}2^{\f{\ell}2}\Bigl(\sum_{k\in\Z}
\|S^\h_{k-1}\dvl a_{\rm hh}\dhk\dvl a_{\rm hh}\|_{L_T^2(L^2_\h(L^1_\vv))}\Bigr)^{\f12}\lesssim \|a\|_{\Bh(T)}.
\eeno
Along the same line, we can prove that the second term in \eqref{ineqB42} shares the same estimate.
This ensures that \eqref{S2eq7} holds for $a_{\rm hh}.$
We thus complete the proof of the proposition.
\end{proof}

\medskip

\setcounter{equation}{0}
\section{Sketch of the proof}

Motivated by the study of the global large solutions to the classical 3-D Navier-Stokes system with
 slowly varying initial data  in  one direction (\cite{CG10, CGZ, CZ15, LZ4}),   here we are going to
decompose the solution of $(ANS)$ as a sum of
a solution to the two-dimensional Navier-Stokes system  with a parameter
and a solution to the three-dimensional perturbed anisotropic Navier-Stokes system.
We point out that compared with the references \cite{CG10, CGZ, CZ15, LZ4}, here the 3-D
 solution to the  perturbed anisotropic Navier-Stokes system will not be small. Indeed only
its vertical component is not small. In order to deal with this part, we are going to appeal to
the observation from \cite{PZ1, Zhang10}, where the authors proved the global well-posedness
 to 3-D anisotropic Navier-Stokes system with the horizontal components of the initial
data being small (see the smallness conditions \eqref{smallPZ1a} and \eqref{smallPZ1}).

For $u^\h=(u^1,u^2),$ we first recall the two-dimensional Biot-Savart's law:
\begin{equation}\label{Helmholtz}
u^\h=u^\h_{\mathop{\rm curl}\nolimits}
+u^\h_{\mathop{\rm div}\nolimits}\with
u^\h_{\mathop{\rm curl}\nolimits}\eqdef\nablah^\perp\Deltah^{-1}(\curlh u^\h)
\andf ~u^\h_{\mathop{\rm div}\nolimits}\eqdef\nablah\Deltah^{-1}(\diveh\uh),
\end{equation} where $\curlh u^\h\eqdefa \pa_1u^2-\pa_2u^1$ and $\diveh\uh\eqdefa \pa_1u^1+\pa_2u^2.$

In particular, let us decompose the horizontal components  $u_0^\h$ of the initial velocity $u_0$  of $(ANS)$ as the sum of
$u^\h_{0,\mathop{\rm curl}\nolimits}$ and $
u^\h_{0,\mathop{\rm div}\nolimits}.$ And we  consider the following 2-D  Navier-Stokes system with a parameter:
\beq \label{S2eq4}
\left\{\begin{array}{l}
\displaystyle \pa_t \baruh +\baruh\cdot\nablah\baruh
-\Delta_\h \baruh=-\nablah \bar p, \qquad (t,x)\in\R^+\times\R^3, \\
\displaystyle \diveh \baruh = 0, \\
\displaystyle  \baruh|_{t=0}=\baruh_0=u^\h_{0,\mathop{\rm curl}\nolimits}.
\end{array}\right.\eeq

Concerning the system \eqref{S2eq4}, we have the following {\it a priori} estimates:

\begin{prop}\label{propbaruh}
{\sl  Let $\baruh_0\in\cB^{0,\f12} $ with  $\Lambda_\h^{-1}\pa_3\baruh_0\in \cB^{0,\f12}.$
Then \eqref{S2eq4} has a unique global solution so that for any time $t>0$, there hold
\begin{equation}\label{baruhestimate1}
\|\baruh\|_{L^\infty_t(\cB^{0,\f12})}
+\|\nablah\baruh\|_{L^2_t(\cB^{0,\f12})}
\leq C\frak{A}_N\bigl(\|\baruh_0\|_{\cB^{0,\f12}}\bigr),
\end{equation}
and
\begin{equation}\label{baruhestimate3}
\begin{split}
\|\Lambda_\h^{-1}\pa_3\baruh\|_{\wt{L}^\infty_t(\cB^{0,\f12})}
+&\|\pa_3\baruh\|_{\wt{L}^2_t(\cB^{0,\f12})}
\leq C\|\Lambda_\h^{-1}\pa_3\baruh_0\|_{\cB^{0,\f12}}
\exp\left(C\frak{A}_N^4\bigl(\|\baruh_0\|_{\cB^{0,\f12}}\bigr)\right),
\end{split}
\end{equation}  where \beq\label{ANuh}
\begin{split}
\baruh_{0,N}\eqdefa& \cF^{-1}\bigl({\bf 1}_{|\xi_3|\leq\frac1N \mbox{or} |\xi_3|\geq N}\cF({\baruh_0})\bigr) \andf\\
\frak{A}_N\bigl(\|\baruh_0\|_{\cB^{0,\f12}}\bigr)\eqdefa& N^{\f12}\|\baruh_0\|_{\cB^{0,\f12}}
\exp\bigl(C\|\baruh_0\|_{\cB^{0,\f12}}^2\bigr)\\
&+\bigl\|\baruh_{0,N}\bigr\|_{\cB^{0,\f12}}\exp\left(N^2\exp\bigl(C\|\baruh_0\|_{\cB^{0,\f12}}^2\bigr)\right),
\end{split}
\eeq and $N$ is taken so large that $\bigl\|\baruh_{0,N}\bigr\|_{\cB^{0,\f12}}$ is sufficiently small.
}\end{prop}

The proof of Proposition \ref{propbaruh} will be presented in
Section \ref{secbaruh}.

\begin{rmk} \label{S3rmk1}
Under the assumptions that $\baruh_0\in L^2$ with $~\pa_3\baruh_0\in L^2$  and $\Lambda_\h^{-1}\pa_3\baruh_0\in \cB^{0,\f12},$ we  have the following alternative estimates
for \eqref{baruhestimate1}  and \eqref{baruhestimate3}
\begin{equation}\label{rmkestimatebaruh}
\|\baruh\|_{\wt{L}^\infty_t(\cB^{0,\f12})}
+\|\nablah\baruh\|_{\wt{L}^2_t(\cB^{0,\f12})}
\leq \|\baruh_0\|_{L^2}^{\f12}\|\pa_3\baruh_0\|_{L^2}^{\f12}
\exp\bigl(C\|\baruh_0\|_{L^2}\|\pa_3\baruh_0\|_{L^2}\bigr),
\end{equation}
and
\begin{equation}\label{rmkestimatebaruh2}
\|\Lambda_\h^{-1}\pa_3\baruh\|_{\wt{L}^\infty_t(\cB^{0,\f12})}
+\|\pa_3\baruh\|_{\wt{L}^2_t(\cB^{0,\f12})}
\leq \|\Lambda_\h^{-1}\pa_3\baruh_0\|_{\cB^{0,\f12}}
\exp\left(\exp\bigl(C\|\baruh_0\|_{L^2}\|\pa_3\baruh_0\|_{L^2}\bigr)\right).
\end{equation} We shall present the proof in Remark \ref{rmkr1r2}.
\end{rmk}

We notice that
\beq\label{S2eq4b}
v_0\eqdefa u_0-\bigl(u^\h_{0,\mathop{\rm curl}\nolimits},0\bigr)=\bigl(u^\h_{0,\mathop{\rm div}\nolimits}, u_0^3\bigr)
\eeq
which satisfies $\dive v_0=0,$ and yet $v_0$ is not small according to our smallness condition \eqref{smallnesscondition}.

Before proceeding, let us recall  the main idea of the proof to Theorem \ref{thmLZ} in \cite{LZ4}.
The authors \cite{LZ4} first  constructed  $(\baruh, \bar{p})$ via the system \eqref{S2eq4}.
Then in order to get rid of the large part of the initial data $v_0,$ given by \eqref{S2eq4b}, the  authors
introduced a correction velocity, $\wt u,$ through the system
 \begin{equation}\label{eqtwtu}
\left\{\begin{array}{l}
\displaystyle \pa_t \wt u +\baruh\cdot\nablah\wt u
-\Delta \wt u=-\nabla \wt p,\\
\displaystyle \dive \wt u = 0, \\
\displaystyle \wt u^\h|_{t=0}=\wt{u}_0^\h=-\nablah\Deltah^{-1}(\pa_3 u^3_0),\quad \wt u^3|_{t=0}
=\wt u^3_0=u^3_0.
\end{array}\right.
\end{equation}
With $\baruh$ and $\wtu$ being determined respectively by the systems \eqref{S2eq4} and \eqref{eqtwtu},
the authors \cite{LZ4} decompose the solution $(u,p)$ to the classical Navier-Stokes system $(NS)$ as
\begin{equation}\label{decomsol}
u= \begin{pmatrix}
\baruh\\
0
\end{pmatrix}+\wt u+v,\quad p=\bar{p}+\wt p+q.
\end{equation}
The key estimate for $v$ states as follows:

\begin{prop}\label{S6prop1}
{\sl Let $u=(u^\h,u^3)\in C([0,T^\ast[; H^{\f12})\cap L^2(]0,T^\ast[; H^{\f32})$ be a Fujita-Kato solution of $(NS).$
 We denote $\omega\eqdefa \pa_1 v^2-\pa_2v^1$ and
\begin{equation}\begin{split}\label{defMN}
&M(t)\eqdef \|\nabla v^3(t)\|_{H^{-\f12,0}}^2
+\|\omega(t)\|_{H^{-\f12,0}}^2
,\quad
N(t)\eqdef \|\nabla^2 v^3(t)\|_{H^{-\f12,0}}^2
+\|\nabla\omega(t)\|_{H^{-\f12,0}}^2.
\end{split}\end{equation} Then under the assumption \eqref{smallLZ}, there exists some
positive constant $\eta$ such that
\beq\label{S6eq4}
\sup_{t\in [0,T^\ast[}\Bigl(M(t)+\int_0^tN(t')\,dt'\Bigr)
\leq \eta.
\eeq
}\end{prop}

Then in order to complete the proof of Theorem \ref{thmLZ}, the authors \cite{LZ4} invoked the following regularity criteria
for the classical Navier-Stokes system:

\begin{thm}[Theorem 1.5 of \cite{CZ5}]
\label{blowupBesovendpoint}
{\sl Let $u\in C([0,T^\ast[; H^{\f12})\cap L^2(]0,T^\ast[; H^{\f32})$ be a solution of~$(NS)$.
 If the maximal existence time $T^\ast$ is finite,
then for any~$(p_{i,j})$ in~$]1,\infty[^9$, one has
\begin{equation}\label{blowupCZ5}
\sum_{1\leq i,j\leq3} \int_0^{T^\ast} \|\partial_{i}
u^{j}(t)\|^{p_{i,j}}_{B_{\infty,\infty}
^{-2+\f2{p_{i,j}}}} \,dt=\infty.
\end{equation}
}\end{thm}

We remark that  Theorem \ref{blowupBesovendpoint} only works
for the classical 3-D Navier-Stokes system. Therefore  the above procedure to prove Theorem \ref{thmLZ}
 can not be applied to construct the global
solutions to the 3-D anisotropic Navier-Stokes system.

On the other hand, we remark that the main observation in \cite{PZ1, Zhang10} is that: by using
$\dive u=0,$ $(ANS)$ can be equivalently reformulated as
\begin{equation*}
(ANS)\quad \left\{\begin{array}{l}
\displaystyle \pa_t u^\h +u^\h\cdot\nabla_\h u^\h+u^3\pa_3u^\h-\Delta_\h u^\h=-\nabla_\h p, \qquad (t,x)\in\R^+\times\R^3, \\
\displaystyle \pa_t u^3 +u^\h\cdot\nabla_3 u^\h-u^3\diveh u^\h-\Delta_\h u^3=-\pa_3 p,\\
\displaystyle \dive u = 0, \\
\displaystyle  u|_{t=0}=(u_0^\h,u_0^3),
\end{array}\right.
\end{equation*}
so that al least seemingly the $u^3$ equation is a linear one. And this explains in some sense
why there is no size restriction for $u_0^3$ in \eqref{smallPZ1a} and \eqref{smallPZ1}.

Motivated by \cite{PZ1, Zhang10}, for $\baruh$ being determined by the systems \eqref{S2eq4},
we  decompose the solution $u$ of $(ANS)$ as $u=\begin{pmatrix}
\baruh\\
0
\end{pmatrix}+v$.
It is easy to verify that the remainder term $v$ satisfies
\begin{equation}\label{eqtv}
\left\{\begin{array}{l}
\displaystyle \pa_t v^\h +v\cdot\nabla v^\h
+\baruh\cdot\nablah v^\h+v\cdot\nabla \baruh
-\Delta_\h v^\h=-\nablah p+\nablah \bar p,\\
\displaystyle \pa_t v^3 +v^\h\cdot\nabla_\h v^3-v^3\diveh v^\h
+\baruh\cdot\nablah v^3
-\Delta_\h v^3=-\pa_3 p,\\
\displaystyle \dive v=0, \\
\displaystyle v|_{t=0}=v_0=\bigl(-\nablah\Deltah^{-1}(\pa_3 u^3_0),u_0^3\bigr).
\end{array}\right.
\end{equation}
We notice that under the smallness condition \eqref{smallnesscondition}, the horizontal
components, $v_0^\h,$ are small in the critical space $\cB^{0,\f12}.$ Then the crucial
ingredient used in the proof of Theorem \ref{thmmain} is that the horizontal components
$v^\h$ of the remainder velocity keeps small for any positive time.

Due to the additional difficulty caused by the fact that $u_0^3$ belongs to the Sobolev-Besov
type space with negative index, as in \cite{CZ07}, we further decompose $v^3$ as
\beq \label{eqv3} v^3=v_F+w,\quad\mbox{where}\quad
v_F(t)\eqdef e^{t\D_\h}u^3_{0,\h\h}\quad\mbox{and}\quad
 u^3_{0,\h\h}\eqdef\sum_{k\geq\ell-1}
\D^\h_k\dvl u^3_0.\eeq
And then $w$ solves
\begin{equation}\label{eqtw}
\left\{\begin{array}{l}
\displaystyle \pa_t w-\Delta_\h w
+v\cdot\nabla v^3+\baruh\cdot\nablah v^3=-\pa_3 p,\\
\displaystyle w|_{t=0}=u^3_{0,{\rm l}\h}\eqdef \sum_{\ell\in\Z}
S^\h_{\ell-1}\dvl u^3_0.
\end{array}\right.
\end{equation}

\begin{prop}\label{aprioriv}
{\sl Let $v$ be a smooth enough solution of \eqref{eqtv} on $[0,T^*[$. Then there exists some positive constant $C$ so that
for any $t\in ]0,T^\ast[,$ we have
\begin{equation}\begin{split}\label{apriorivh}
\|\vh&\|_{\infB}
+\bigl(\f54-C\|\vh\|_{\infB}^{\f12}\bigr)\|\nablah\vh\|_\twoB
\leq\bigl(\|\vh_0\|_{\cB^{0,\f12}}+\|\pa_3\baruh\|_\twoB\bigr)\\
&\qquad\ \times
\exp\Bigl(C\int_0^t\bigl(\|w(t')\|_{\cB^{0,\f12}}^2\|\nablah w(t')\|_{\cB^{0,\f12}}^2
+\|\baruh(t')\|_{\cB_4^{0,\f12}}^4+\|v_F(t')\|_{\cB_4^{0,\f12}}^4\bigr)\,dt'\Bigr),
\end{split}\end{equation}
 and
\begin{equation}\begin{split}\label{aprioriv3}
\Bigl(\f56&-C\bigl(\|\vh\|_{\cB^{0,\f12}(t)}^{\f12}
+\|\pa_3\baruh\|_\twoB^{\f12}\bigr)\Bigr)\|w\|_{\cB^{0,\f12}(t)}\\
\leq& \|u^3_0\|_{\cB^{-\f12,\f12}_4}+C\Bigl(\|\vh\|_{\cB^{0,\f12}(t)}+\|\pa_3\baruh\|_\twoB+\|\vh\|_{\cB^{0,\f12}(t)}^2
\\
&+\bigl(1+
\|\vh\|_{\cB^{0,\f12}(t)}+\|\pa_3\baruh\|_\twoB\bigr)\|v_F\|_{\cB^{-\f12,\f12}_4(t)}\Bigr)\exp\Bigl(C\|\baruh\|_{L^4_t(\cB_4^{0,\f12})}^4\Bigr).
\end{split}\end{equation}
}\end{prop}

The proof of the estimates \eqref{apriorivh} and \eqref{aprioriv3} will be presented respectively in Sections \ref{secapriorivh} and \ref{secaprioriv3}.
Now let us admit the above Propositions \ref{propbaruh} and \ref{aprioriv}
temporarily, and continue our proof of Theorem \ref{thmmain}.

\begin{proof}[Proof of Theorem \ref{thmmain}]
It is well-known that the existence of global solutions to a nonlinear partial differential equations
can be obtained by first constructing the approximate solutions, and then performing uniform estimates and finally passing to the limit
to such approximate solutions. For simplicity, here we just present the {\it a priori} estimates for smooth enough solutions of
$(ANS).$

Let $u$ be a smooth enough solution of $(ANS)$ on $[0, T^\ast[$ with $T^\ast$ being the maximal time of existence.
  Let $\baruh$ and $v$ be determined by \eqref{S2eq4} and \eqref{eqtv} respectively. Thanks to \eqref{Helmholtz} and Proposition \ref{propbaruh},
we first take $L, M, N$ large enough and $\e_0$ small enough in \eqref{smallnesscondition} so that
\beq
\label{S2eq20}
\begin{split}
\|\Lambda_\h^{-1}\pa_3\baruh\|_{\wt{L}^\infty_t(\cB^{0,\f12})}
+\|\pa_3\baruh\|_{\wt{L}^2_t(\cB^{0,\f12})}\leq &C\|\Lambda_\h^{-1}\pa_3\uh_0\|_{\cB^{0,\f12}}
\exp\Bigl(C\frak{A}_N^4\bigl(\|u^\h_0\|_{\cB^{0,\f12}}\bigr)
\Bigr)\\
\leq &\f1{16}\quad\mbox{for any }\ \ t>0.
\end{split}
\eeq
We now define
\beq \label{S2eq18}
T^\star\eqdefa \sup\Bigl\{\ t< T^\ast,\ \ C\|v^\h\|_{\cB^{0,\f12}(t)}\leq \f1{16}\ \Bigr\}.
\eeq
Then thanks to \eqref{S2eq20} and Proposition \ref{aprioriv}, for $t\leq T^\star,$ we find
\beq \label{S2eq21}
\begin{split}
\|&\vh\|_{\cB^{0,\f12}(t)}
\leq\bigl(\|\Lambda_\h^{-1}\pa_3 u^3_0\|_{\cB^{0,\f12}}+\|\pa_3\baruh\|_\twoB\bigr)\\
&\quad \times\exp\Bigl(C\int_0^t \bigl(\|w(t')\|_{\cB^{0,\f12}}^2\|\nablah w(t')\|_{\cB^{0,\f12}}^2
+\|\baruh(t')\|_{\cB_4^{0,\f12}}^4+\|v_F(t')\|_{\cB_4^{0,\f12}}^4\bigr)\,dt'\Bigr),
\end{split}\eeq
and
\begin{equation}\begin{split}\label{S2eq19}
\f13\|w\|_{\cB^{0,\f12}(t)}
\leq &\|u^3_0\|_{\cB^{-\f12,\f12}_4}+C\bigl(1+\|v_F\|_{\Bh(t)}\bigr)
\exp\Bigl(C\|\baruh\|_{L^4_t(\cB_4^{0,\f12})}^4\Bigr).
\end{split}\end{equation}
It follows from Lemma \ref{lemCZ07} and Proposition \ref{lemB4} that
\begin{equation*}\label{5.33}
\|v_F\|_{L^4_t(\cB_4^{0,\f12})}\lesssim\|v_F\|_{\Bh(t)}
\lesssim \|u^3_0\|_\Bh,
\end{equation*}
Whereas we deduce from \eqref{S2eq5a} and Proposition \ref{propbaruh} that
\beno
\begin{split}
\|\baruh\|_{\wt L^4_t(\cB_4^{0,\f12})}^4\leq &C\|\baruh\|_{\wt L^\infty_t(\cB^{0,\f12})}^2\|\na_\h\baruh\|_{\wt L^2_t(\cB^{0,\f12})}^2\\
\leq & C\frak{A}_N^4\bigl(\|u^\h_0\|_{\cB^{0,\f12}}\bigr).
\end{split}
\eeno
By inserting the above two inequalities to \eqref{S2eq19} and using \eqref{baruhestimate1}, we obtain that for $t\leq T^\star$
\beq \label{S2eq22}
\begin{split}
\f13\|w\|_{\cB^{0,\f12}(t)}\leq C\bigl(1+\|u_0^3\|_{\cB^{-\f12,\f12}_4}\bigr)\exp\Bigl( C\frak{A}_N^4\bigl(\|u^\h_0\|_{\cB^{0,\f12}}\bigr)\Bigr).
\end{split}
\eeq
Then we deduce that for $t\leq T^\star,$
\begin{align*}
\int_0^t& \bigl(\|w(t')\|_{\cB^{0,\f12}}^2\|\nablah w(t')\|_{\cB^{0,\f12}}^2
+\|\baruh(t')\|_{\cB_4^{0,\f12}}^4+\|v_F(t')\|_{\cB_4^{0,\f12}}^4\bigr)\,dt'\\
&\leq \|w\|_{L^\infty_t(\cB^{0,\f12})}^2
\|\nablah w\|_{L^2_t(\cB^{0,\f12})}^2
+\|\baruh\|_{L^4_t(\cB_4^{0,\f12})}^4+\|v_F\|_{L^4_t(\cB_4^{0,\f12})}^4\\
&\leq C\bigl(1+\|u_0^3\|_{\cB^{-\f12,\f12}_4}^4\bigr)\exp\Bigl(  C\frak{A}_N^4\bigl(\|u^\h_0\|_{\cB^{0,\f12}}\bigr)\Bigr).
\end{align*}
Inserting the above estimates into \eqref{S2eq21} gives rise to
\beq \label{S2eq23}
\|v^\h\|_{\cB^{0,\f12}(t)}\leq \|\La_\h^{-1}\pa_3u_0\|_{\cB^{0,\f12}}\exp\Bigl( C\bigl(1+\|u_0^3\|_{\cB^{-\f12,\f12}_4}^4\bigr)\exp\bigl( C\frak{A}_N^4\bigl(\|u^\h_0\|_{\cB^{0,\f12}}\bigr)\bigr)\Bigr)
\eeq for $t\leq T^\star.$ Therefore, if we take $L, M, N$ large enough and $\e_0$ small enough in \eqref{smallnesscondition},
we deduce from \eqref{S2eq23} that
\beq \label{S2eq24}
C\|v^\h\|_{\cB^{0,\f12}(t)}\leq \f1{32}\quad\mbox{for} \ \ t\leq T^\star.
\eeq
\eqref{S2eq24} contradicts with \eqref{S2eq18}. This in turn shows that $T^\star=T^\ast.$
\eqref{S2eq22} along with \eqref{S2eq24} shows that $T^\ast=\infty.$ Moreover, thanks to \eqref{eqv3}, we have
$\frak{v}\eqdefa u-e^{t\D_h} \begin{pmatrix}
0\\
u^3_{0,{\rm hh}}
\end{pmatrix}\in C([0,\infty[\,;\cB^{0,\f12})$ with $\nablah \frak{v}\in L^2([0,\infty[\,;\cB^{0,\f12}).$
This completes the proof of our Theorem \ref{thmmain}.
\end{proof}

\begin{proof}[Proof of Corollary \ref{S1col1}] Under the assumptions that $\uh_0\in L^2$ with $~\pa_3\uh_0\in L^2$  and $\Lambda_\h^{-1}\pa_3\uh_0\in \cB^{0,\f12},$ we deduce from \eqref{Helmholtz}, \eqref{baruhestimate3} and \eqref{rmkestimatebaruh2} that
\beno\begin{split}
&\|\baruh\|_{\wt{L}^\infty_t(\cB^{0,\f12})}
+\|\nablah\baruh\|_{\wt{L}^2_t(\cB^{0,\f12})}
\leq \|\uh_0\|_{L^2}^{\f12}\|\pa_3\uh_0\|_{L^2}^{\f12}
\exp\bigl(C\|\uh_0\|_{L^2}\|\pa_3\uh_0\|_{L^2}\bigr),\\
&
\|\Lambda_\h^{-1}\pa_3\baruh\|_{\wt{L}^\infty_t(\cB^{0,\f12})}
+\|\pa_3\baruh\|_{\wt{L}^2_t(\cB^{0,\f12})}
\leq \|\Lambda_\h^{-1}\pa_3\uh_0\|_{\cB^{0,\f12}}
\exp\left(\exp\bigl(C\|\uh_0\|_{L^2}\|\pa_3\uh_0\|_{L^2}\bigr)\right).
\end{split}
\eeno
Then by repeating the argument from \eqref{S2eq20} to \eqref{S2eq23}, we conclude the proof of  Corollary \ref{S1col1}.
\end{proof}

\medskip
\setcounter{equation}{0}
\section{Estimates of the 2-D solution $\baruh$}\label{secbaruh}

The goal of  this section is to present the proof of Proposition \ref{propbaruh}. Let us start the proof by the following
lemma, which is in the spirit of Lemma 3.1 of \cite{CG10}.

\begin{lem}\label{lem4.1}
{\sl Let $\ah=(a^1,a^2)$ be a smooth enough solution of
\begin{equation}\label{eqtah}
\left\{\begin{array}{l}
\displaystyle \pa_t \ah +\ah\cdot\nablah\ah
-\Delta_\h \ah=-\nablah \pi,
\qquad (t,x)\in\R^+\times\R^3, \\
\displaystyle \diveh \ah = 0,\\
\displaystyle  \ah|_{t=0}=\ah_0.
\end{array}\right.
\end{equation}
 Then for any $t>0$ and any fixed $x_3\in\R,$
there holds
\begin{equation}\label{4.1}
\|\ah(t,\cdot,x_3)\|_{L^2_\h}^2
+2\int_0^t\|\nablah\ah(t',\cdot,x_3)\|_{L^2_\h}^2dt'
=\|\ah_0(\cdot,x_3)\|_{L^2_\h}^2,
\end{equation}
and
\beq\label{4.1a}
\begin{split}
\|\pa_3\ah(t,\cdot,x_3)\|_{L^2_\h}^2
+&\int_0^t\|\nablah \pa_3\ah(t',\cdot,x_3)\|_{L^2_\h}^2\,dt'
\leq \|\pa_3\ah_0(\cdot,x_3)\|_{L^2_\h}^2
\exp\bigl(C\|\ah_0\|_{L^\infty_\v(L^2_\h)}^2\bigr). \end{split}
\eeq}
\end{lem}

\begin{proof}
By taking  $L^2_\h$ inn-product of \eqref{eqtah} with $\ah$
and using $\diveh \ah=0,$  we obtain \eqref{4.1}.

While by applying  $\pa_3$ to \eqref{eqtah}
and then taking $L^2_\h$ inner product of the resulting equation with
$\pa_3\ah,$ we find
\begin{equation}\label{4.2}
\begin{split}
\f12\f{d}{dt}\|\pa_3\ah(t,\cdot,x_3)\|_{L^2_\h}^2
+&\|\nablah \pa_3\ah(t,\cdot,x_3)\|_{L^2_\h}^2\\
=&-\bigl(\pa_3(\ah\cdot\nablah\ah)(t,\cdot,x_3) \big| \pa_3\ah(t,\cdot,x_3)\bigr)_{L^2_\h}.
\end{split}
\end{equation}
Due to $\diveh \ah=0,$ we get, by applying \eqref{S2eq8}, that
\begin{align*}
\bigl|\bigl(&\pa_3(\ah\cdot\nablah\ah)(t,\cdot,x_3) | \pa_3\ah(t,\cdot,x_3)\bigr)_{L^2_\h}\bigr|\\
&=\bigl|\bigl((\pa_3\ah\cdot\na_\h\ah)(t,\cdot,x_3) | \pa_3\ah(t,\cdot,x_3)\bigr)_{L^2_\h}\bigr|\\
&\leq\|\nablah\ah(t,\cdot,x_3)\|_{L^2_\h} \|\pa_3\ah(t,\cdot,x_3)\|_{L^4_\h}^2\\
&\leq C\|\nablah\ah(t,\cdot,x_3)\|_{L^2_\h} \|\pa_3\ah(t,\cdot,x_3)\|_{L^2_\h}\|\na_\h\pa_3\ah(t,\cdot,x_3)\|_{L^2_\h}.
\end{align*}
Applying Young's inequality yields
\begin{align*}
\bigl|\bigl(\pa_3(\ah\cdot\nablah\ah)(t,\cdot,x_3) |& \pa_3\ah(t,\cdot,x_3)\bigr)_{L^2_\h}\bigr|\\
\leq &\f12\|\nablah\pa_3\ah(t,\cdot,x_3)\|_{L^2_\h}^2
+C\|\nablah\ah(t,\cdot,x_3)\|_{L^2_\h}^2\|\pa_3\ah(t,\cdot,x_3)\|_{L^2_\h}^2.
\end{align*}
Inserting the above estimate into \eqref{4.2} gives rise to
\begin{equation*}
\f{d}{dt}\|\pa_3\ah(t,\cdot,x_3)\|_{L^2_\h}^2
+\|\nablah \pa_3\ah(t,\cdot,x_3)\|_{L^2_\h}^2
\leq C\|\nablah\ah(t,\cdot,x_3)\|_{L^2_\h}^2\|\pa_3\ah(t,\cdot,x_3)\|_{L^2_\h}^2.
\end{equation*}
Applying Gronwall's inequality and using \eqref{4.1}, we achieve
\begin{align*}
\|\pa_3\ah(t,\cdot,x_3)\|_{L^2_\h}^2
&+\int_0^t\|\nablah \pa_3\ah(t',\cdot,x_3)\|_{L^2_\h}^2\,dt'\\
&\leq\|\pa_3\ah_0(\cdot,x_3)\|_{L^2_\h}^2
\exp\Bigl(C\int_0^t
\|\nablah\ah(t',\cdot,x_3)\|_{L^2_\h}^2\,dt'\Bigr)\\
&\leq\|\pa_3\ah_0(\cdot,x_3)\|_{L^2_\h}^2
\exp\bigl(C\|\ah_0(\cdot,x_3)\|_{L^2_\h}^2\bigr),
\end{align*}
which leads to \eqref{4.1a}. This completes the proof of this lemma.
\end{proof}

Let us now present the proof of Proposition \ref{propbaruh}.

\begin{proof}[Proof of Proposition \ref{propbaruh}] For any positive integer $N,$ and $\baruh_{0,N}$  being given by \eqref{ANuh},
we split the solution $\baruh$ to \eqref{S2eq4} as
\beq \label{S4eq1a}
\baruh=\baruh_1+\baruh_2,
\eeq
with $\baruh_1$ and $\baruh_2$ being determined respectively by
\begin{equation}\label{eqtbarvh}
\left\{\begin{array}{l}
\displaystyle \pa_t \baruh_1 +\baruh_1\cdot\nablah\baruh_1
-\Delta_\h \baruh_1=-\nablah \bar p^{(1)},
\qquad (t,x)\in\R^+\times\R^3, \\
\displaystyle \diveh \baruh_1 = 0,\\
\displaystyle \baruh_1|_{t=0}=\baruh_{1,0}\eqdefa\baruh_0-\baruh_{0,N},
\end{array}\right.
\end{equation}
and
\begin{equation}\label{eqtbarwh}
\left\{\begin{array}{l}
\displaystyle \pa_t \baruh_2 +\diveh\bigl(\baruh_2\otimes\baruh_2
+\baruh_1\otimes\baruh_2+\baruh_2\otimes\baruh_1\bigr)
-\Delta_\h \baruh_2=-\nablah \bar p^{(2)},\\
\displaystyle \diveh \baruh_2 = 0,\\
\displaystyle \baruh_2|_{t=0}=\baruh_{2,0}=\baruh_{0,N}.
\end{array}\right.
\end{equation}

We first deduce from Lemma \ref{lem4.1} that
\beno
\begin{split}
\|\baruh_1(t)\|_{L^2}^2
+2\int_0^t\|\nablah\baruh_1(t')\|_{L^2}^2dt'
=&\|\baruh_{1,0}\|_{L^2}^2\lesssim
N \|\baruh_0\|_{\cB^{0,\f12}}^2,\quad\mbox{and}\\
\|\pa_3\baruh_1(t)\|_{L^2}^2
+\int_0^t\|\nablah \pa_3\baruh_1(t')\|_{L^2}^2\,dt'
&\leq \|\pa_3\baruh_{1,0}\|_{L^2}^2
\exp\bigl(C\|\baruh_{1,0}\|_{L^\infty_\v(L^2)}^2\bigr)\\
&\lesssim N\|\baruh_0\|_{\cB^{0,\f12}}^2
\exp\bigl(C\|\baruh_0\|_{\cB^{0,\f12}}^2\bigr), \end{split}
\eeno
which together with \eqref{interLpcB} ensures that
\begin{equation}\label{4.6}
\|\baruh_1\|_{\wt{L}^\infty_t(\cB^{0,\f12})}
+\|\nablah\baruh_1\|_{\wt{L}^2_t(\cB^{0,\f12})}
\leq C N^{\f12}\|\baruh_0\|_{\cB^{0,\f12}}
\exp\bigl(C\|\baruh_0\|_{\cB^{0,\f12}}^2\bigr).
\end{equation}

Next we handle the estimate of $\baruh_2$. To do it,
 for any $\kappa>0,$ we denote
\beq\label{4.3a} f^\h(t)\eqdefa\|\baruh_1(t)\|_{\cB^{0,\f12}}^2\|\na_\h\baruh_2(t)\|_{\cB^{0,\f12}}^2 \andf \baruh_{2,\kappa}(t)\eqdefa \baruh_2(t)
\exp\Bigl(-\kappa \int_0^t f^\h(t')\,dt'\Bigr).
\eeq
Then by multiplying $\exp\Bigl(-\kappa\int_0^t f^\h(t')\,dt'\Bigr)$  to the
$\baruh_2$ equation in \eqref{eqtbarwh}, we write
\beno
\pa_t\baruh_{2,\ka}+\ka f^\h(t)\baruh_{2,\ka}-\D_\h \baruh_{2,\ka}
+\diveh(\baruh_2\otimes\baruh_{2,\ka}
+\baruh_1\otimes\baruh_{2,\ka}+\baruh_{2,\ka}\otimes\baruh_1)=-\na_\h \bar{p}^{(2)}_\ka.
\eeno
Applying the operator $\D_\ell^\v$ to the above equation and taking $L^2$ inner product of the resulting equation with
$\D_\ell^\v\baruh_{2,\ka},$ and then using integration by parts, we get
\begin{equation}\label{4.9}
\begin{split}
\f12\f{d}{dt}&\|\D_\ell^\v\baruh_{2,\ka}(t)\|_{L^2}^2+\ka f^\h(t)\|\D_\ell^\v\baruh_{2,\ka}(t)\|_{L^2}^2
+\|\D_\ell^\v\na_\h\baruh_{2,\ka}\|_{L^2}^2\\
&=-\bigl(\D_\ell^\v(\baruh_2\cdot\nablah\baruh_{2,\ka})
\big| \D_\ell^\v\baruh_{2,\ka}\bigr)_{L^2}
+\bigl(\D_\ell^\v(\baruh_1\otimes\baruh_{2,\ka}+\baruh_{2,\ka}\otimes\baruh_1)
\big| \D_\ell^\v\nablah\baruh_{2,\ka}\bigr)_{L^2}.
\end{split}
\end{equation}

The estimate of the second line of \eqref{4.9} relies on
 the following lemma, whose proof will be postponed in the Appendix \ref{appenda}:
\begin{lem}\label{lemfone}
{\sl Let $a,b,c\in\cB^{0,\f12}(T)$ and $\ff(t)\eqdefa\|a(t)\|_{\cB_4^{0,\f12}}^4$.
Then for any smooth homogeneous Fourier multiplier, $A(D),$ of degree zero
and any $\ell\in\Z$, there hold
\ben
\label{fone1}
&&\int_0^T\bigl|\bigl(\dvl A(D)(a\otimes b)\big|\dvl c\bigr)_{L^2}\bigr|\,dt
\lesssim  d_{\ell}^2 2^{-\ell} \|b\|_\twoBT
\|c\|_\twofBT^{\f12}\|\nablah c\|_\twoBT^{\f12},\\
\label{fone2}
&&\int_0^T\bigl|\bigl(\dvl A(D)(a\otimes b)\big|\dvl c\bigr)_{L^2}\bigr|\,dt
\lesssim  d_{\ell}^2 2^{-\ell}
\|b\|_\twofBT^{\f12}\|\nablah b\|_\twoBT^{\f12}\|c\|_\twoBT.
\een
Moreover, for non-negative function $\fg\in L^\infty(0,T),$ one has
\begin{equation}\begin{split}\label{fone3}
\int_0^T\bigl|\bigl(\dvl A(D)(a\otimes b)\big|\dvl c\bigr)_{L^2}\bigr|
\cdot\fg^2\,dt
\lesssim& d_{\ell}^2 2^{-\ell}\|a\|_\infBT^{\f12}\|\fg\nablah a\|_\twoBT^{\f12}\\
&\times\|\fg b\|_\twoBT
\|c\|_\infBT^{\f12}\|\fg\nablah c\|_\twoBT^{\f12}.
\end{split}\end{equation}
}\end{lem}
By applying \eqref{fone3} with $a=c=\baruh_2,$  $b=\na_h\baruh_{2,\kappa}$
and $\ff=
\exp\Bigl(-\kappa\int_0^t f^\h(t')\,dt'\Bigr)$, we get
\begin{equation}\label{4.13}
\int_0^t\bigl|\bigl(\D_\ell^\v(\baruh_2\cdot\nablah\baruh_{2,\ka})
\big| \D_\ell^\v\baruh_{2,\ka}\bigr)_{L^2}\bigr|\,dt'
\lesssim d_\ell^22^{-\ell}\|\baruh_2\|_{\wt L^\infty_t(\cB^{0,\f12})}
\|\nablah\baruh_{2,\ka}\|_{\wt L^2_{t}(\cB^{0,\f12})}^2.
\end{equation}
Whereas due to \eqref{S2eq5}, one has
$$\|\baruh_1(t)\|_{\cB_4^{0,\f12}}^4\lesssim
\|\baruh_1(t)\|_{\cB^{0,\f12}}^2\|\nablah \baruh_1(t)\|_{\cB^{0,\f12}}^2.$$
By applying \eqref{fone2} with $a=\baruh_1,~b=\baruh_{2,\ka},~
c=\nablah\baruh_{2,\ka}$,  we infer
\begin{equation}\label{4.14}
\begin{split}
\int_0^t\bigl|\bigl(\D_\ell^\v(\baruh_1\otimes\baruh_{2,\ka}+\baruh_{2,\ka}\otimes\baruh_1)
\big|&\D_\ell^\v\nablah\baruh_{2,\ka}\bigr)_{L^2}\bigr|\,dt'\\
\lesssim& d_\ell^22^{-\ell}\|\baruh_{2,\ka}\|_{\wt L^2_{t,f^\h}(\cB^{0,\f12})}^{\f12}
\|\nablah\baruh_{2,\ka}\|_{\wt L^2_{t}(\cB^{0,\f12})}^{\f32}.
\end{split}
\end{equation}
Then we get, by  first integrating \eqref{4.9} over $[0,t]$ and inserting \eqref{4.13} and \eqref{4.14} into the resulting inequality, that
\begin{align*}
&\|\D_\ell^\v\baruh_{2,\ka}(t)\|_{L^2}^2+2\ka \int_0^tf^\h(t')\|\D_\ell^\v\baruh_{2,\ka}(t')\|_{L^2}^2\,dt'
+2\|\D_\ell^\v\nablah\baruh_{2,\ka}\|_{L^2_t(L^2)}^2\\
&\leq\|\D_\ell^\v\baruh_{0,N}\|_{L^2}^2
+C d_\ell^22^{-\ell}\Bigl(\|\baruh_{2}\|_{\wt L^\infty_t(\cB^{0,\f12})}
\|\nablah\baruh_{2,\ka}\|_{\wt L^2_{t}(\cB^{0,\f12})}^2\\
&\qquad\qquad\qquad\qquad\qquad\qquad+\|\baruh_{2,\ka}\|_{\wt L^2_{t,f^\h}(\cB^{0,\f12})}^{\f12}
\|\nablah\baruh_{2,\ka}\|_{\wt L^2_t(\cB^{0,\f12})}^{\f32}\Bigr).
\end{align*}
Multiplying the above inequality by $2^\ell$ and taking square root of the resulting inequality, and then summing up the inequalities
for $\ell\in\Z,$  we arrive at
\begin{align*}
&\|\baruh_{2,\ka}\|_{\wt{L}^\infty_t(\cB^{0,\f12})}
+\sqrt{2\ka}\|\baruh_{2,\ka}\|_{\wt L^2_{t,f^\h}(\cB^{0,\f12})}
+\sqrt2\|\nablah\baruh_{2,\ka}\|_{\wt{L}^2_t(\cB^{0,\f12})}\\
&\leq\|\baruh_{0,N}\|_{\cB^{0,\f12}}
+C\Bigl(\|\baruh_{2}\|_{\wt L^\infty_t(\cB^{0,\f12})}^{\f12}
\|\nablah\baruh_{2,\ka}\|_{\wt L^2_{t}(\cB^{0,\f12})}
+\|\baruh_{2,\ka}\|_{\wt L^2_{t,f^\h}(\cB^{0,\f12})}^{\f14}
\|\nablah\baruh_{2,\ka}\|_{\wt L^2_t(\cB^{0,\f12})}^{\f34}\Bigr)\\
&\leq\|\baruh_{0,N}\|_{\cB^{0,\f12}}
+\bigl(\sqrt2-1+C\|\baruh_{2}\|_{\wt L^\infty_t(\cB^{0,\f12})}^{\f12}\bigr)
\|\nablah\baruh_{2,\ka}\|_{\wt L^2_{t}(\cB^{0,\f12})}
+C\|\baruh_{2,\ka}\|_{\wt L^2_{t,f^\h}(\cB^{0,\f12})}.
\end{align*}
In particular, taking $2\ka=C^2$ in the above inequality gives rise to
\begin{equation} \label{4.15}
\|\baruh_{2,\ka}\|_{\wt{L}^\infty_t(\cB^{0,\f12})}
+\bigl(1-C\|\baruh_{2}\|_{\wt L^\infty_t(\cB^{0,\f12})}^{\f12}\bigr)
\|\nablah\baruh_{2,\ka}\|_{\wt{L}^2_t(\cB^{0,\f12})}
\leq\|\baruh_{0,N}\|_{\cB^{0,\f12}}.
\end{equation}
On the other hand, in view of \eqref{ANuh}, we can take $N$ so large that
\beq \label{S4eq2a} C\|\baruh_{0, N}\|_{\cB^{0,\f12}}^{\f12}
\leq \f12.\eeq
Then a standard continuity argument shows that,  for any time $t>0$,
there holds
\begin{equation} \label{4.16}
\|\baruh_{2,\ka}\|_{\wt{L}^\infty_t(\cB^{0,\f12})}
+\f12\|\nablah\baruh_{2,\ka}\|_{\wt{L}^2_t(\cB^{0,\f12})}
\leq \|\baruh_{0,N}\|_{\cB^{0,\f12}}.
\end{equation}
Due to the definition of $\baruh_{2,\la}$ given by \eqref{4.3a},
one has
\begin{equation*}
\begin{split}
\bigl(\|\baruh_{2}\|_{\wt{L}^\infty_t(\cB^{0,\f12})}
+\|\nablah\baruh_{2}\|_{\wt{L}^2_t(\cB^{0,\f12})}\bigr)&\exp\Bigl(-\ka\int_0^tf^\h(t')\,dt'\Bigr)\\
\leq&
\|\baruh_{2,\ka}\|_{\wt{L}^\infty_t(\cB^{0,\f12})}
+\|\nablah\baruh_{2,\ka}\|_{\wt{L}^2_t(\cB^{0,\f12})},
\end{split}
\end{equation*}
which together with \eqref{4.6} and  \eqref{4.16} implies that
\begin{equation}\label{4.17}
\begin{split}
\|\baruh_2\|_{\wt{L}^\infty_t(\cB^{0,\f12})}
+\|\nablah\baruh_2\|_{\wt{L}^2_t(\cB^{0,\f12})}
\leq & 2\|\baruh_{0,N}\|_{\cB^{0,\f12}}
\exp\Bigl(\ka\int_0^tf^\h(t')\,dt'\Bigr)\\
\leq & 2\|\baruh_{0,N}\|_{\cB^{0,\f12}}\exp\left(N^2\exp\bigl(C\|\baruh_0\|_{\cB^{0,\f12}}^2\bigr)\right),
\end{split}\end{equation}

By combining \eqref{4.6} with \eqref{4.17}, we obtain \eqref{baruhestimate1}.

It remains to prove \eqref{baruhestimate3}. In order to do,
for any $\gamma>0,$  we  denote
\beq\label{4.19} g^\h(t)\eqdefa\|\baruh(t)\|_{\cB^{0,\f12}}^2\|\na_\h\baruh(t)\|_{\cB^{0,\f12}}^2 \andf \baruh_\gamma(t)\eqdefa \baruh(t)
\exp\Bigl(-\gamma\int_0^t g^\h(t')\,dt'\Bigr).
\eeq
Then by multiplying $\exp\left(-\gamma\int_0^t g^\h(t')\,dt'\right)$  to the
$\baruh$ equation in \eqref{S2eq4}, we write
\beno
\pa_t\baruh_\ga+\ga g^\h(t)\baruh_\ga-\D_\h \baruh_\ga
+\baruh\cdot\nablah\baruh_\ga=-\na_\h \bar{p}_\ga.
\eeno
Applying the operator $\D_\ell^\v\Lambda_\h^{-1}\pa_3$ to the above equation
and then taking $L^2$ inner product of the resulting equation with
$\D_\ell^\v\Lambda_\h^{-1}\pa_3\baruh_\ga,$ we get
\begin{equation}\label{4.20}
\begin{split}
\f12\f{d}{dt}\|\D_\ell^\v\Lambda_\h^{-1}\pa_3\baruh_\ga(t)\|_{L^2}^2
&+\ga g^\h(t)\|\D_\ell^\v\Lambda_\h^{-1}\pa_3\baruh_\ga(t)\|_{L^2}^2
+\|\D_\ell^\v\na_\h\Lambda_\h^{-1}\pa_3\baruh_\ga\|_{L^2}^2\\
&=-\bigl(\D_\ell^\v\Lambda_\h^{-1}\pa_3(\baruh\cdot\nablah\baruh_\ga)
\big| \D_\ell^\v\Lambda_\h^{-1}\pa_3\baruh_\ga\bigr)_{L^2}\\
&=-\bigl(\D_\ell^\v\Lambda_\h^{-1}\diveh(\baruh\otimes\pa_3\baruh_\ga
+\pa_3\baruh_\ga\otimes\baruh)
\big| \D_\ell^\v\Lambda_\h^{-1}\pa_3\baruh_\ga\bigr)_{L^2}.
\end{split}
\end{equation}
Noting that $\Lambda_\h^{-1}\diveh$ is a bounded Fourier multiplier,
we get, by using \eqref{fone1} with $a=\baruh,~b=
\pa_3\baruh_\ga$ and $c=\Lambda_\h^{-1}\pa_3\baruh_\ga,$ that
\begin{align*}
\int_0^t\bigl|
\bigl(\D_\ell^\v\Lambda_\h^{-1}\diveh(\baruh\otimes\pa_3\baruh_\ga
&+\pa_3\baruh_\ga\otimes\baruh)
\big| \D_\ell^\v\Lambda_\h^{-1}\pa_3\baruh_\ga\bigr)_{L^2}\bigr|\,dt'\\
&\lesssim d_\ell^22^{-\ell}\|\pa_3\baruh_\ga\|_{\wt{L}^2_t(\cB^{0,\f12})}^{\f32}
\|\Lambda_\h^{-1}\pa_3\baruh_\ga\|_{\wt{L}^2_{t,g^\h}(\cB^{0,\f12})}^{\f12}.
\end{align*}
By integrating \eqref{4.20} over $[0,t]$ and then inserting the above estimate into the resulting inequality, we find
\begin{equation*}
\begin{split}
\|\D_\ell^\v\Lambda_\h^{-1}\pa_3\baruh_\ga(t)\|_{L^2}^2&+2\ga \int_0^tg^\h(t')\|\D_\ell^\v\Lambda_\h^{-1}\pa_3\baruh_\ga(t')\|_{L^2}^2\,dt'
+2\|\D_\ell^\v\pa_3\baruh_\ga\|_{L^2_t(L^2)}^2\\
&\leq \|\D_\ell^\v\Lambda_\h^{-1}\pa_3\baruh_0\|_{L^2}^2
+C d_\ell^22^{-\ell}\|\pa_3\baruh_\ga\|_{\wt{L}^2_t(\cB^{0,\f12})}^{\f32}
\|\Lambda_\h^{-1}\pa_3\baruh_\ga\|_{\wt{L}^2_{t,g^\h}(\cB^{0,\f12})}^{\f12}.
\end{split}
\end{equation*}
Multiplying the above inequality by $2^\ell$ and taking square root of the resulting inequality, and then summing up the inequalities
for $\ell\in\Z,$ we arrive at
\beno
\begin{split}
\|\Lambda_\h^{-1}&\pa_3\baruh_\ga\|_{\wt{L}^\infty_t(\cB^{0,\f12})}
+\sqrt{2\ga}\|\Lambda_\h^{-1}\pa_3\baruh_\la\|_{\wt{L}^2_{t,f^\h}(\cB^{0,\f12})}
+\sqrt2\|\pa_3\baruh_\ga\|_{\wt{L}^2_t(\cB^{0,\f12})}\\
\leq &\|\Lambda_\h^{-1}\pa_3\baruh_0\|_{\cB^{0,\f12}}+C \|\pa_3\baruh_\ga\|_{\wt{L}^2_t(\cB^{0,\f12})}^{\f34}
\|\Lambda_\h^{-1}\pa_3\baruh_\ga\|_{\wt{L}^2_{t,g^\h}(\cB^{0,\f12})}^{\f14}\\
\leq &\|\Lambda_\h^{-1}\pa_3\baruh_0\|_{\cB^{0,\f12}}+(\sqrt2-1) \|\pa_3\baruh_\ga\|_{\wt{L}^2_t(\cB^{0,\f12})}+C
\|\Lambda_\h^{-1}\pa_3\baruh_\ga\|_{\wt{L}^2_{t,g^\h}(\cB^{0,\f12})}.
\end{split}
\eeno
In particular, taking $2\ga=C^2$ in the above inequality gives rise to
$$
\|\Lambda_\h^{-1}\pa_3\baruh_\ga\|_{\wt{L}^\infty_t(\cB^{0,\f12})}
+\|\pa_3\baruh_\ga\|_{\wt{L}^2_t(\cB^{0,\f12})}
\leq \|\Lambda_\h^{-1}\pa_3\baruh_0\|_{\cB^{0,\f12}}.
$$
Then a similar derivation from \eqref{4.16} to \eqref{4.17} leads to
\beq \label{S4eq4}
\|\Lambda_\h^{-1}\pa_3\baruh\|_{\wt{L}^\infty_t(\cB^{0,\f12})}
+\|\pa_3\baruh\|_{\wt{L}^2_t(\cB^{0,\f12})}
\leq\|\Lambda_\h^{-1}\pa_3\baruh_0\|_{\cB^{0,\f12}}
\exp\Bigl(\ga\int_0^tg^\h(t')\,dt'\Bigr),\eeq
which together with \eqref{baruhestimate1} ensures \eqref{baruhestimate3}.
This completes the proof of this proposition.
\end{proof}

\begin{rmk}\label{rmkr1r2}
For smoother initial data $\baruh_0,$ we may write explicitly the constant $\frak{A}_N\bigl(\|\baruh_0\|_{\cB^{0,\f12}}\bigr)$ in \eqref{baruhestimate1}. For instance,
if $\baruh_0\in L^2$ with $~\pa_3\baruh_0\in L^2$ and $\Lambda_\h^{-1}\pa_3\baruh_0\in \cB^{0,\f12}$,
we deduce from
Lemma \ref{lem4.1} that
\beno
\begin{split}
&\|\baruh(t)\|_{L^2}^2
+2\int_0^t\|\nablah\baruh(t')\|_{L^2}^2dt'
=\|\baruh_0\|_{L^2}^2,\quad\mbox{and}\\
\|\pa_3\baruh&(t)\|_{L^2}^2
+\int_0^t\|\nablah \pa_3\baruh(t')\|_{L^2}^2\,dt'
\leq \|\pa_3\baruh_0\|_{L^2}^2
\exp\bigl(C\|\baruh_0\|_{L^\infty_\v(L^2)}^2\bigr), \end{split}
\eeno
 which together with \eqref{interLpcB} and
$$\|\baruh_0\|_{L^\infty_\v(L^2_\h)}^2\leq \|\baruh_0\|_{L^2_\h(L^\infty_\v)}^2\leq \|u_0^\h\|_{\cB^{0,\f12}}^2
\leq \|u^\h_0\|_{L^2}\|\pa_3u^\h_0\|_{L^2},$$
 ensures \eqref{rmkestimatebaruh}. By virtue of  \eqref{rmkestimatebaruh} and \eqref{S4eq4},
 we deduce  \eqref{rmkestimatebaruh2}.
\end{rmk}

\medskip
\setcounter{equation}{0}
\section{The estimate of the horizontal components $v^\h$}\label{secapriorivh}

The goal of this section is to present the proof of \eqref{apriorivh}, namely, we are going
to deal with the estimate to the horizontal components of the remainder velocity determined by \eqref{eqtv}.

In order to do so, let $u$ be a smooth enough solution of $(ANS)$ on $[0,T^\ast[,$ let $\baruh, v_F$ and $w$
be determined respectively by \eqref{S2eq4}, \eqref{eqv3} and \eqref{eqtw},
for any constant $\lam>0$, we  denote
\begin{equation}\label{5.1}
\begin{split}
&\vh_{\lam}(t)\eqdef\vh(t)\exp\Bigl(-\lam\int_0^tf(t')\,dt'\Bigr)\with\\
f(t)\eqdefa &\|w(t)\|_{\cB^{0,\f12}}^2\|\nablah w(t)\|_{\cB^{0,\f12}}^2
+\|\baruh(t)\|_{\cB_4^{0,\f12}}^4+\|v_F(t)\|_{\cB_4^{0,\f12}}^4,
\end{split}
\end{equation}
and similar notations for
$\baruh_\lam,~p_\lam,~\bar p_\lam$ and $\vh_{\lam/2}$.

By multiplying $\exp\Bigl(-\lam\int_0^tf(t')\,dt'\Bigr)$ to the $\vh$ equation of
 \eqref{eqtv}, we get
$$\pa_t \vhl+\lam f(t)\vhl +v\cdot\nabla \vhl
+\baruh\cdot\nablah \vhl+v_\lam\cdot\nabla \baruh
-\Delta_\h \vhl=-\nablah p_\lam+\nablah \bar p_\lam.
$$
Applying $\dvl$ to the above equation and taking $L^2$ inner product of the
resulting equation with $\dvl\vhl$, and then integrating the equality over $[0,t],$ we obtain
\begin{equation}\label{5.2}
\f12\|\dvl\vhl(t)\|_{L^2}^2+\lam\int_0^t f(t')\|\dvl\vhl\|_{L^2}^2\,dt'
+\int_0^t\|\nablah\dvl\vhl\|_{L^2}^2\,dt'
=\f12\|\dvl\vh_0\|_{L^2}^2
-\sum_{i=1}^6{\rm I}_i,
\end{equation}
where
\begin{align*}
&{\rm I}_1\eqdefa \int_0^t\bigl(\dvl(\baruh\cdot\nablah \vhl)
\,\big|\,\dvl\vhl\bigr)_{L^2}\,dt',\quad
{\rm I}_2\eqdefa\int_0^t\bigl(\dvl(v^\h\cdot\nablah \vhl)
\,\big|\,\dvl\vhl\bigr)_{L^2}\,dt',\\
&{\rm I}_3\eqdefa\int_0^t\bigl(\dvl(\vh_\lam\cdot\nablah \baruh)
\,\big|\,\dvl\vhl\bigr)_{L^2}\,dt',\quad
{\rm I}_4\eqdefa\int_0^t\bigl(\dvl (v^3\pa_3 \baruh_\lam)
\,\big|\,\dvl\vhl\bigr)_{L^2}\,dt',\\
&{\rm I}_5\eqdefa\int_0^t\bigl(\dvl(v^3\pa_3 \vhl)
\,\big|\,\dvl\vhl\bigr)_{L^2}\,dt',\qquad
{\rm I}_6\eqdefa\int_0^t\bigl(\dvl\nablah(p_\lam-\bar p_\lam)
\,\big|\,\dvl\vhl\bigr)_{L^2}\,dt'.
\end{align*}
We mention that since our system \eqref{eqtv} has
only horizontal dissipation, it is reasonable to
distinguish the terms above with horizontal derivatives
from the ones with vertical derivative. Next let us handle term by term above.

\noindent$\bullet$\underline{
The estimates of ${\rm I}_1$ to ${\rm I}_4.$}

We first get, by using \eqref{fone1} with $a=\baruh,~b=\nablah\vhl$
and $c=\vhl,$ that
\begin{equation}\label{5.3}
|{\rm I}_1|\lesssim d_{\ell}^2 2^{-\ell}\|\vhl\|_\twofB^{\f12}
\|\nablah\vhl\|_\twoB^{\f32}.
\end{equation}
Applying \eqref{fone3} with $a=\vh,~b=\nablah\vh,~c=\vh$
and $\fg(t)=\exp\bigl(-\lam\int_0^tf(t')\,dt'\bigr)$ yields
\begin{equation}\label{5.4}
|{\rm I}_2|
\lesssim d_\ell^2 2^{-\ell}\|\vh\|_{\infB}\|\nablah\vhl\|_\twoB^2.
\end{equation}

To handle ${\rm I}_3,$
by using integration by parts, we write
\begin{equation*}\label{5.5}
{\rm I}_3=-\int_0^t\bigl(\dvl(\diveh\vhl\cdot \baruh)
\big|\dvl\vhl\bigr)_{L^2}\,dt'
-\int_0^t\bigl(\dvl(\baruh\otimes\vhl)
\big|\dvl\nablah\vhl\bigr)_{L^2}\,dt'.
\end{equation*}
Applying \eqref{fone1} with $a=\baruh,~b=\diveh\vhl$
and $c=\vhl$ gives
$$
\Bigl|\int_0^t\bigl(\dvl(\diveh\vhl\cdot \baruh)
\big|\dvl\vhl\bigr)_{L^2}\,dt'\Bigr|
\lesssim d_{\ell}^2 2^{-\ell}\|\vhl\|_\twofB^{\f12}
\|\nablah\vhl\|_\twoB^{\f32}.
$$
Whereas applying \eqref{fone2} with
$a=\baruh,~b=\vhl$ and $c=\nablah\vhl$ yields
$$
\Bigl|\int_0^t\bigl(\dvl(\baruh\otimes\vhl)
\big|\dvl\nablah\vhl\bigr)_{L^2}\,dt'\Bigr|
\lesssim d_{\ell}^2 2^{-\ell}\|\vhl\|_\twofB^{\f12}
\|\nablah\vhl\|_\twoB^{\f32}.
$$
As a result, it comes out
\begin{equation}\label{5.6}
|{\rm I}_3|\lesssim d_{\ell}^2 2^{-\ell}\|\vhl\|_\twofB^{\f12}
\|\nablah\vhl\|_\twoB^{\f32}.
\end{equation}

While by applying \eqref{fone1}
with $a=v^3,~b=\pa_3 \baruh_\lam,~c=\vhl$, and using the fact that
\beno
\|v^3(t)\|_{\cB^{0,\f12}_4}\leq \|v_F(t)\|_{\cB^{0,\f12}_4}+\|w(t)\|_{\cB^{0,\f12}}^{\f12}\|\na_\h w(t)\|_{\cB^{0,\f12}}^{\f12},\eeno
we find
\begin{equation}\label{5.7}
|{\rm I}_4|\lesssim d_{\ell}^2 2^{-\ell}\|\vhl\|_\twofB^{\f12}
\|\nablah\vhl\|_\twoB^{\f12}\|\pa_3\baruh_\lam\|_\twoB.
\end{equation}

\noindent$\bullet$\underline{
The estimates of ${\rm I}_5$.}

The estimate of ${\rm I}_5$ is much more complicated,
since there is no vertical dissipation in $(ANS)$.
To overcome this difficulty, we first use Bony's decomposition
in vertical variable \eqref{bony} to write
$${\rm I}_5=\int_0^t\bigl(\D_\ell^\v\bigl(T^\v_{v^3}\pa_3 \vhl
+ R^\v({v^3},\pa_3 \vhl) \bigr)\,\big|\,\dvl\vhl\bigr)_{L^2}
\,dt'\eqdefa {\rm I}_5^{T}+{\rm I}_5^{R}.$$
Following \cite{CZ07, Pa02}, we get, by using a standard commutator's process, that
\begin{align*}
{\rm I}_5^{T}=\sum_{|\ell'-\ell|\leq 5}\Bigl(&\int_0^t\bigl([\dvl; S^\v_{\ell'-1}v^3]\dvlp\pa_3\vhl \,\big|\,\dvl\vhl\bigr)_{L^2}\,dt'
\\
&+\int_0^t\bigl((S^\v_{\ell'-1}v^3-S^\v_{\ell-1}v^3)
\D_\ell^\v\dvlp\pa_3\vhl \,\big|\,\dvl\vhl\bigr)_{L^2}\,dt'\Bigr)\\
+\int_0^t&\bigl(S^\v_{\ell-1}v^3\dvl\pa_3\vhl \,\big|\,\dvl\vhl\bigr)_{L^2}\,dt'
\eqdef{\rm I}_5^{T,1}+{\rm I}_5^{T,2}+{\rm I}_5^{T,3}.
\end{align*}
By applying  commutator's estimate
(see Lemma $2.97$ in \cite{BCD}), we find
\begin{align*} \bigl|{\rm I}_5^{T,1}\bigr|
&\leq
\sum_{|\ell'-\ell|\leq 5}\|[\dvl; S^\v_{\ell'-1}v^3_\lam]
\dvlp\pa_3 \vh_{\lam/2}\|_{L^{\f43}_t(L^{\f43}_\h(L^2_\v))}
\|\dvl \vh_{\lam/2}\|_{L^4_t(L^4_\h(L^2_\v))}\\
&\lesssim\sum_{|\ell'-\ell|\leq 5}2^{-\ell}
\|\pa_3S^\v_{\ell'-1}v^3_\lam\|_{L^2_t(L^2_\h(L^\infty_\v))}
\|\dvlp\pa_3 \vh_{\lam/2}\|_{L^4_t(L^4_\h(L^2_\v))}
\|\dvl \vh_{\lam/2}\|_{L^4_t(L^4_\h(L^2_\v))}.
\end{align*}
Due to $\pa_3v^3=-\diveh v^\h,$ we get, by applying \eqref{S2eq8}, that
\begin{equation*}\begin{split}\label{5.10}
  \bigl|{\rm I}_5^{T,1}\bigr|
&\lesssim\sum_{|\ell'-\ell|\leq5}2^{-\ell}
\|S^\v_{\ell'-1}\diveh\vhl\|_{L^2_t(L^2_\h(L^\infty_\v))}
2^{\ell'}\|\dvlp \vh_{\lam/2}\|_{L^4_t(L^4_\h(L^2_\v))}
\|\dvl\vh_{\lam/2}\|_{L^4_t(L^4_\h(L^2_\v))}\\
&\lesssim\sum_{|\ell'-\ell|\leq5} \|\nablah\vhl\|_\twoB
\|\dvlp \vh\|_{L^\infty_t(L^2)}^{\f12}\|\na_\h\dvlp \vh_\la\|_{L^2_t(L^2)}^{\f12}\\
&\qquad\qquad\qquad\qquad\qquad\qquad\qquad\qquad\times\|\D_\ell^\v \vh\|_{L^\infty_t(L^2)}^{\f12}
\|\na_\h\D_\ell^\v \vh_\la\|_{L^2_t(L^2)}^{\f12}\\
&\lesssim
d_\ell^2 2^{-\ell}\|\vh\|_\infB\|\nablah \vh_\la\|_\twoB^2.
\end{split}\end{equation*}
Next,
 since the support to the Fourier transform of
$\sum_{|\ell'-\ell|\leq5}
(S^\v_{\ell'-1}v^3-S^\v_{\ell-1}v^3)$ is contained in $\R^2\times \cup_{|\ell'-\ell|\leq 5}2^{\ell'}\cC_\v,$ we get, by applying Lemma \ref{lemBern}, that
\beno
\begin{split}
\bigl|{\rm I}_5^{T,3}\bigr|
\lesssim & \sum_{|\ell'-\ell|\leq 5}2^{-\ell}\|\pa_3(S^\v_{\ell'-1}v^3_\lam-S^\v_{\ell-1}v^3_\lam)\|_{L^2_t(L^2_\h(L^\infty_\v))}
\|\dvlp\pa_3\vh_{\lam/2}\|_{L^4_t(L^4_\h(L^2_\v))}\|\dvl \vh_{\lam/2}\|_{L^4_t(L^4_\h(L^2_\v))},
\end{split}
\eeno
from which, we infer
\beno
\begin{split}
\bigl|{\rm I}_5^{T,3}\bigr|
&\lesssim d_\ell^2 2^{-\ell}\|\vh\|_\infB\|\nablah\vhl\|_\twoB^2.
\end{split}
\eeno
As a result, it comes out
\begin{equation}\label{5.9}
\bigl|{\rm I}_5^{T}\bigr|\lesssim d_\ell^2 2^{-\ell}\|\vh\|_\infB
\|\nablah\vhl\|_\twoB^2.
\end{equation}
Finally,  by using integration by parts and $\pa_3v^3=-\diveh v^\h$
again,  we find
\begin{equation*}\begin{split}\label{5.8}
\bigl|{\rm I}_5^{T,3}\bigr|&=\f12\Bigl|\int_0^t\int_{\R^3}
S^\v_{\ell-1}\pa_3 v^3_\lam\cdot\bigl|\dvl\vh_{\lam/2}\bigr|^2\,dxdt'\Bigr|\\
&\lesssim\|S^\v_{\ell-1}\diveh\vhl\|_{L^2_t(L^2_\h(L^\infty_\v))}
\|\dvl \vh_{\lam/2}\|_{L^4_t(L^4_\h(L^2_\v))}^2\\
&\lesssim d_\ell^2 2^{-\ell}\|\vh\|_\infB\|\nablah\vhl\|_\twoB^2.
\end{split}\end{equation*}

On the other hand, by applying Lemma \ref{lemBern} once again, we find
\begin{align*}
\bigl|{\rm I}_5^{R}\bigr|&\lesssim
\sum_{\ell'\geq\ell-4}
\|\dvlp v^3_\lam\|_{L^2_t(L^2)}
2^{\ell'}\|S^\v_{\ell'+2} \vh_{\lam/2}\|_{L^4_t(L^4_\h(L^\infty_\v))}
\|\dvl \vh_{\lam/2}\|_{L^4_t(L^4_\h(L^2_\v))}\\
&\lesssim
\sum_{\ell'\geq\ell-4}
\|\pa_3\dvlp v^3_\lam\|_{L^2_t(L^2)}
\|S^\v_{\ell'+2}\vh_{\lam/2}\|_{L^4_t(L^4_\h(L^\infty_\v))}
\|\dvl\vh_{\lam/2}\|_{L^4_t(L^4_\h(L^2_\v))}.
\end{align*}
Observing that
\beno
\begin{split}
\|\pa_3\dvlp v^3_\lam\|_{L^2_t(L^2)}
\lesssim &  d_{\ell'}2^{-\f{\ell'}2}\|\diveh\vhl\|_\twoB,\\
\|S^\v_{\ell'+2}\vh_{\lam/2}\|_{L^4_t(L^4_\h(L^\infty_\v))}
\lesssim &\|\vh\|_\infB^{\f12}\|\nablah \vh_\lam\|_\twoB^{\f12},
\end{split}
\eeno
we infer
\begin{equation*}\label{5.11}
\bigl|{\rm I}_5^{R}\bigr|\lesssim
d_\ell^2 2^{-\ell}\|\vh\|_\infB
\|\nablah\vhl\|_\twoB^2,
\end{equation*}
which together with \eqref{5.9} ensures that
\begin{equation}\label{5.12}
|{\rm I}_{5}|\lesssim
d_\ell^2 2^{-\ell}\|\vh\|_\infB\|\nablah\vhl\|_\twoB^2.
\end{equation}

\noindent$\bullet$\underline{
The estimates of ${\rm I}_6.$}

We first get, by taking the space divergence operators, $\dive$ and $\diveh,$ to $(ANS)$ and \eqref{S2eq4} respectively, that
\beq \label{5.12a} -\D p=\dive(u\cdot\nabla u)\quad\mbox{and}\quad
-\D_\h \bar p=\diveh(\baru\cdot\nablah \baru),\eeq
so that thanks to the fact that $$
u=(u^\h,u^3)=(\baruh,0)+(v^\h,v^3),$$ we write
\beno
\begin{split}
\na_\h p-\na_\h\bar{p}=&\na_\h(-\D)^{-1}\diveh(v\cdot\na u^\h+\bar{u}^\h\cdot\na_\h v^\h)\\
&+\nablah(-\D)^{-1}\pa_3(u\cdot\na v^3)\\
&+\na_\h\bigl((-\D)^{-1}-(-\D_\h)^{-1}\bigr)\diveh\diveh\bigl(\baruh\otimes\baruh\bigr).
\end{split}
\eeno
Accordingly, we  decompose ${\rm I}_6$ as
$${\rm I}_6=
{\rm I}_{6,1}+{\rm I}_{6,2}+{\rm I}_{6,3}+{\rm I}_{6,4},$$ where
\beno
\begin{split}
{\rm I}_{6,1}=&\int_0^t\bigl(\dvl\nablah(-\D)^{-1}\diveh
\bigl(\baruh\cdot\nablah\vhl
+v^\h\cdot\nabla_\h\vhl+v_\lam\cdot\nabla\baruh\bigr)\,\big|\,\dvl\vhl\bigr)_{L^2}\,dt',\\
{\rm I}_{6,2}=&\int_0^t\bigl(\dvl\nablah(-\D)^{-1}\diveh
(v^3\pa_3 v^\h_\la)\,\big|\,\dvl\vhl\bigr)_{L^2}\,dt',\\
{\rm I}_{6,3}=&\int_0^t\bigl(\dvl\nablah(-\D)^{-1}\pa_3
\bigl(v_\lam\cdot\nabla v^3+\baruh\cdot\nablah v^3_\lam
\bigr)\,\big|\,\dvl\vhl\bigr)_{L^2}\,dt',\\
{\rm I}_{6,4}=&\sum_{i=1}^2\sum_{j=1}^2\int_0^t\bigl(\dvl\nablah
\bigl((-\D)^{-1}-(-\D_\h)^{-1}\bigr)
\pa_i\pa_j(\baru^i \baru^j_\lam)\,\big|\,\dvl\vhl\bigr)_{L^2}\,dt'.
\end{split}
\eeno

Noticing that $\nablah(-\D)^{-1}\diveh$ is a bounded Fourier multiplier.
Then along the same line to the estimate of ${\rm I}_1$ to ${\rm I}_4,$ we achieve
\begin{equation}\begin{split}\label{5.13}
|{\rm I}_{6,1}|\lesssim d_{\ell}^2 2^{-\ell}
\Bigl(&\|\vh\|_{\infB}\|\nablah\vhl\|_\twoB^2\\
&+\|\vhl\|_\twofB^{\f12}\|\nablah\vhl\|_\twoB^{\f12}
\bigl(\|\nablah\vhl\|_\twoB+\|\pa_3\baruh_\lam\|_\twoB\bigr)\Bigr).
\end{split}\end{equation}

However, ${\rm I}_{6,2}$ can not be handled along the same line to that of
${\rm I}_5$, since the symbol of the operator $\nablah(-\D)^{-1}\diveh$
depends  not only on $\xi_3$, but also on $\xi_\h,$ which makes it impossible for us to
deal with the commutator's estimate.
Fortunately, the appearance of the operator $(-\D)^{-1}$ can absorb the vertical derivative. Indeed,
by using integration by parts, and the divergence-free condition of $v,$ we write
\begin{align*}
{\rm I}_{6,2}&=\int_0^t\Bigl(\dvl\nablah(-\D)^{-1}\diveh
\bigl(\pa_3(v^3 v^\h_\la)-\pa_3 v^3\cdot\vhl\bigr)
\,\Big|\,\dvl\vhl\Bigr)_{L^2}\,dt'\\
&=-\int_0^t\bigl(\dvl\nablah(-\D)^{-1}\pa_3
( v^3 v^\h_\la)\,\big|\,\dvl\nablah\vhl\bigr)_{L^2}\,dt'\\
&\qquad\qquad+\int_0^t\bigl(\dvl\nablah(-\D)^{-1}\diveh
(\diveh\vh\cdot \vhl)\,\big|\,\dvl\vhl\bigr)_{L^2}\,dt'.
\end{align*}
Since both $\nablah(-\D)^{-1}\pa_3$ and $\nablah(-\D)^{-1}\diveh$
are bounded Fourier multiplier,  we get, by applying Lemma \ref{lemfone}, that
\beq \label{5.13a}
|{\rm I}_{6,2}|\lesssim d_{\ell}^2 2^{-\ell}
\bigl(\|\vhl\|_\twofB^{\f12}\|\nablah\vhl\|_\twoB^{\f32}
+\|\vh\|_{\infB}\|\nablah\vhl\|_\twoB^2\bigr).
\eeq

To handle ${\rm I}_{6,3}$, we use $\dive v=\diveh\baruh=0$
 to write
\begin{align*}
{\rm I}_{6,3}=&\int_0^t\bigl(\dvl\nablah(-\D)^{-1}\pa_3\dive
(v_\lam v^3)+\dvl\nablah(-\D)^{-1}\pa_3\diveh
(\baruh v^3_\lam)\,\big|\,\dvl\vhl\bigr)_{L^2}\,dt'\\
=&\int_0^t\left(\nablah(-\D)^{-1}\dvl\bigl(\diveh
(v^3\pa_3\vh_\lam+\vh\pa_3 v^3_\lam)+2\pa_3(v^3\pa_3v^3_\lam)\bigr) \,\big|\,\dvl\vhl\right)_{L^2}\,dt' \\
&+\int_0^t\left(\nablah(-\D)^{-1}\diveh\dvl
\bigl(v^3\pa_3\baruh_\lam+\baruh\pa_3 v^3_\lam\bigr)\,\big|\,\dvl\vhl\right)_{L^2}\,dt'\\
=&\int_0^t\Bigl(\nablah(-\D)^{-1}\dvl\bigl(\diveh
(v^3\pa_3\vh_\lam-\vh\diveh\vhl)-2\pa_3(v^3\diveh\vhl)\bigr)\,\big|\,\dvl\vhl\Bigr)_{L^2}\,dt'\\
&+\int_0^t\Bigl(\nablah(-\D)^{-1}\diveh\dvl
\bigl(v^3\pa_3\baruh_\lam-\baruh\diveh\vhl\bigr)\,\big|\,\dvl\vhl\Bigr)_{L^2}\,dt'.
\end{align*}
Applying \eqref{fone1} with $A(D)=\nablah(-\D)^{-1}\pa_3,
a=v^3, b=\diveh v^\h_\la$ and $c=v^\h_\la$ yields
\beno
\int_0^t\bigl|\bigl(\nablah(-\D)^{-1}\pa_3\dvl(v^3\diveh\vhl)\,\big|\,\dvl\vhl\bigr)_{L^2}\bigr|\,dt'
\lesssim d_{\ell}^2 2^{-\ell}
\|\vhl\|_\twofB^{\f12}\|\nablah\vhl\|_\twoB^{\f32}.
\eeno
The remaining terms in ${\rm I}_{6,3}$ can be handled along the same line to that of ${\rm I}_{6,1}$ and ${\rm I}_{6,2}.$
As a consequence, we obtain
\begin{equation}\label{5.14}
\begin{split}
|{\rm I}_{6,3}|\lesssim d_{\ell}^2 2^{-\ell}
\Bigl(&\|\vh\|_{\infB}\|\nablah\vhl\|_\twoB^2\\
&+\|\vhl\|_\twofB^{\f12}\|\nablah\vhl\|_\twoB^{\f12}\bigl(\|\nablah\vhl\|_\twoB
+\|\pa_3\baruh_\lam\|_\twoB\bigr)\Bigr).
\end{split}
\end{equation}

To deal with ${\rm I}_{6,4}$,  it is crucial to observe that
$$\dvl\nablah\bigl((-\D)^{-1}-(-\D_\h)^{-1}\bigr)
\pa_i\pa_j(\baru^i \baru^j_\lam)
=\dvl\nablah\pa_3^2(-\D)^{-1}(-\D_\h)^{-1}
\pa_i\pa_j(\baru^i \baru^j_\lam).$$
Then due to the fact that $\sum_{i,j=1}^2\nablah\pa_3(-\D)^{-1}(-\D_\h)^{-1}
\pa_i\pa_j$ is a bounded Fourier multiplier, we  get,  by applying \eqref{fone1}
with $a=\baruh,b=\pa_3\baruh_\lam,c=\vhl,$ that
\begin{equation}\begin{split}\label{5.15}
|{\rm I}_{6,4}|&\leq 2\sum_{i=1}^2\sum_{j=1}^3\int_0^t\bigl|\bigl(\dvl\nablah
\pa_3(-\D)^{-1}(-\D_\h)^{-1}
\pa_i\pa_j(\baru^i\pa_3 \baru^j_\lam)\,\big|\,\dvl\vhl\bigr)_{L^2}\bigr|\,dt'\\
&\lesssim d_{\ell}^2 2^{-\ell}\|\vhl\|_\twofB^{\f12}
\|\nablah\vhl\|_\twoB^{\f12}\|\pa_3\baruh_\lam\|_\twoB.
\end{split}\end{equation}

By summing up (\ref{5.13}-\ref{5.15}), we arrive at
\begin{equation}\begin{split}\label{5.15a}
|{\rm I}_{6}|\lesssim d_{\ell}^2 2^{-\ell}
\Bigl(&\|\vh\|_{\infB}\|\nablah\vhl\|_\twoB^2\\
&+\|\vhl\|_\twofB^{\f12}\|\nablah\vhl\|_\twoB^{\f12}
\bigl(\|\nablah\vhl\|_\twoB+\|\pa_3\baruh_\lam\|_\twoB\bigr)\Bigr).
\end{split}\end{equation}

Now we are in a position to complete the proof of \eqref{apriorivh}.

\begin{proof}[Proof of \eqref{apriorivh}]
By inserting the estimates (\ref{5.3}-\ref{5.7}),
 \eqref{5.12} and \eqref{5.15a} into \eqref{5.2}, we achieve
\begin{equation*}\begin{split}
\f12\|\dvl\vhl(t)\|_{L^2}^2+\lam&\int_0^t f(t')\|\dvl\vhl(t')\|_{L^2}^2\,dt'
+\int_0^t\|\nablah\dvl\vhl(t')\|_{L^2}^2\,dt'\\
\leq &\f12\|\dvl\vh_0\|_{L^2}^2+C d_{\ell}^2 2^{-\ell}
\Bigl(\|\vh\|_{\infB}\|\nablah\vhl\|_\twoB^2\\
&\qquad+\|\vhl\|_\twofB^{\f12}\bigl(\|\nablah\vhl\|_\twoB^{\f32}
+\|\pa_3\baruh_\lam\|_\twoB^{\f32}\bigr)\Bigr).
\end{split}\end{equation*}
Multiplying the above inequality by $2^{\ell+1}
$  and taking square root
of the resulting inequality, and then summing up the inequalities over $\Z,$ we find
\begin{equation}\begin{split}\label{5.16}
\|\vhl\|_{\infB}+&\sqrt{2\lam}\|\vhl\|_{\twofB}
+\sqrt2\|\nablah\vhl\|_\twoB\\
\leq &\|\vh_0\|_{\cB^{0,\f12}}+C\|\vh\|_{\infB}^{\f12}\|\nablah\vhl\|_\twoB\\
&+C \|\vhl\|_\twofB^{\f14}\bigl(\|\nablah\vhl\|_\twoB^{\f34}
+\|\pa_3\baruh_\lam\|_\twoB^{\f34}\bigr).
\end{split}\end{equation}
It follows from Young's inequality  that
\begin{align*}
C \|\vhl\|_\twofB^{\f14}&\bigl(\|\nablah\vhl\|_\twoB^{\f34}
+\|\pa_3\baruh_\lam\|_\twoB^{\f34}\bigr)\\
\leq&\f1{10}\|\nablah\vhl\|_\twoB
+\|\pa_3\baruh_\lam\|_\twoB+C\|\vhl\|_\twofB.
\end{align*}
Inserting the above inequality into \eqref{5.16} and taking $\lam$ so that
 $\sqrt{2\lam}=C$, we obtain
\beno
\begin{split}
\|\vhl\|_{\infB}+%\|\vhl\|_{\twofB}
\f54\|\nablah\vhl\|_\twoB
\leq &\|\vh_0\|_{\cB^{0,\f12}}+\|\pa_3\baruh_\lam\|_\twoB\\
&+C\|\vh\|_{\infB}^{\f12}\|\nablah\vhl\|_\twoB, \end{split}
\eeno
which together with the following  consequence of
\eqref{5.1} that
$$\|a\|_{\wt L^p_t(\cB^{0,\f12})}\exp\Bigl(-\lam\int_0^tf(t')\,dt'\Bigr)
\leq\|a_\lam\|_{\wt L^p_t(\cB^{0,\f12})}\quad \mbox{for}
\quad p=2\ \mbox{or}\ \infty,$$
gives rise to \eqref{apriorivh}. \end{proof}

\medskip
\setcounter{equation}{0}
\section{The estimate of the vertical component $v^3$}\label{secaprioriv3}

The purpose of this section is to present the proof of \eqref{aprioriv3}. Compared with \cite{PZ1}, where
the third component of the velocity field can be estimated in the standard Besov spaces, here due to the
additional terms like $\baruh\cdot\na_\h v$ appears in \eqref{eqtv},
we will have to  use the weighted Chemin-Lerner norms once again.
Indeed  for any constant $\mu>0$, we denote
\begin{equation}\label{5.17}
w_{\mu}(t)\eqdef w(t)\bfg(t)
\quad\mbox{with}\quad \bfg(t)\eqdefa\exp\Bigl(-\mu\int_0^t\hbar(t')\,dt'\Bigr)\andf
\hbar(t)\eqdefa \|\baruh(t)\|_{\cB_4^{0,\f12}}^4,
\end{equation}
And similar notations for $v_\mu,~\baruh_\mu,$ and $~p_\mu$.

By multiplying $\bfg(t)$ to
\eqref{eqtw}, we write
$$\pa_t \wg+\mu \hbar(t)\wg-\Delta_\h \wg
+v\cdot\nabla v^3_\mu
+\baruh\cdot\nablah v^3_\mu
=-\pa_3 p_\mu.$$
By applying $\dvl$ to the above equation and taking $L^2$ inner product of the
resulting equation with $\dvl\wg$, and then integrating the equality over $[0,t],$  we obtain
\begin{equation}\label{5.18}
\f12\|\dvl\wg(t)\|_{L^2}^2+\mu\|\sqrt{\hbar}\dvl\wg\|_{L^2_t(L^2)}^2
+\|\nablah\dvl\wg\|_{L^2_t(L^2)}^2
=\f12\|\dvl u^3_{0,{\rm l}\h}\|_{L^2}^2
-\sum_{i=1}^6{\rm II}_i,
\end{equation}
where
\begin{align*}
&{\rm II}_1\eqdefa\int_0^t\bigl(\dvl(\baruh\cdot\nablah \wg)
\,\big|\,\dvl\wg\bigr)_{L^2}\,dt',\quad
{\rm II}_2\eqdefa\int_0^t\bigl(\dvl(\vh\cdot\nablah \wg)
\,\big|\,\dvl\wg\bigr)_{L^2}\,dt',\\
&{\rm II}_3\eqdefa\int_0^t\bigl(\dvl(v^\h_\mu\cdot\nablah v_F)
\,\big|\,\dvl\wg\bigr)_{L^2}\,dt',\quad
{\rm II}_4\eqdefa\int_0^t\bigl(\dvl(\baruh_\mu\cdot\nablah v_F)
\,\big|\,\dvl\wg\bigr)_{L^2}\,dt',\\
&{\rm II}_5\eqdefa\int_0^t\bigl(\dvl(v^3\pa_3 v^3_\mu)
\,\big|\,\dvl\wg\bigr)_{L^2}\,dt',\qquad\
{\rm II}_6\eqdefa\int_0^t\bigl(\dvl\pa_3 p_\mu
\,\big|\,\dvl\wg\bigr)_{L^2}\,dt'.
\end{align*}

Let us handle term by term above.

\noindent$\bullet$\underline{
The estimates  of ${\rm II}_1$ and ${\rm II}_2$}

We first get, by applying \eqref{fone1} with $a=\baruh,~b=\nablah\wg$
and $c=\wg,$ that
\begin{equation}\label{5.19}
|{\rm II}_1|\lesssim d_{\ell}^2 2^{-\ell}\|\wg\|_\twogB^{\f12}
\|\nablah\wg\|_\twoB^{\f32}.
\end{equation}
Whereas by applying a modified version of \eqref{fone3} with
$a=\vh,~b=\nablah w_\mu,~c=w_\mu$
and $\fg(t)=\exp\bigl(-\mu\int_0^t\hbar(t')\,dt'\bigr),$ we find
\begin{equation}\label{5.20}
|{\rm II}_2|
\lesssim d_\ell^2 2^{-\ell}\|\vh\|_{\infB}^{\f12}\|\nablah\vh\|_\twoB^{\f12}
\|\wg\|_{\infB}^{\f12}\|\nablah\wg\|_\twoB^{\f32}.
\end{equation}

\noindent$\bullet$\underline{
The estimate  of ${\rm II}_3$}

The estimate of ${\rm II}_3$ relies on the following lemma, the  proof of which
will be postponed in the Appendix \ref{appenda}.

\begin{lem}\label{lembhcb}
{\sl Let $a,~c\in\cB^{0,\f12}(T) $ and $b\in\Bh(T).$
Then for any smooth homogeneous Fourier multiplier, $A(D),$ of degree zero
and any $\ell\in\Z$, there hold
\begin{equation}\label{bhcb1}
\int_0^T\bigl|\bigl(A(D)\dvl(a\otimes b)\big|\dvl c\bigr)_{L^2}\bigr|
\,dt'
\lesssim d_{\ell}^2 2^{-\ell}\|a\|_{\wt L^4_T(\cB_4^{0,\f12})}
\|b\|_{\Bh(T)}\|c\|_\twoBT,
\end{equation}
and
\begin{equation}\label{bhcb1p}
\int_0^T\bigl|\bigl(A(D)\dvl(a\otimes b)\big|\dvl c\bigr)_{L^2}
\bigr|\,dt'
\lesssim d_{\ell}^2 2^{-\ell}\|a\|_\twoBT
\|b\|_{\Bh(T)}\|c\|_{\cB^{0,\f12}(T)}.
\end{equation}
}\end{lem}

\begin{rmk}\label{S6rmk1}
Indeed the proof of Lemma \ref{lembhcb} shows that $\|b\|_{\Bh(T)}$ in \eqref{bhcb1} and \eqref{bhcb1p}
can be replaced by $\|b\|_{\cB^{0,\f12}(T)}.$
\end{rmk}

Let us admit this lemma temporarily, and  continue our estimate of ${\rm II}_3$.
By using integration by parts, we write
\begin{equation}\label{5.21}
{\rm II}_3=-\int_0^t\bigl(\dvl(\diveh v^\h_\mu\cdot v_F)
\big|\dvl\wg\bigr)_{L^2}\,dt'
-\int_0^t\bigl(\dvl(v^\h_\mu\otimes v_F)
\big|\dvl\nablah\wg\bigr)_{L^2}\,dt'.
\end{equation}

Applying \eqref{bhcb1p} with $a=\diveh v^\h_\mu,~b=v_F$
and $c=\wg$ yields
\begin{equation}\label{5.22}
\Bigl|\int_0^t\bigl(\dvl(\diveh\vhg\cdot v_F)
\big|\dvl\wg\bigr)_{L^2}\,dt'\Bigr|
\lesssim d_{\ell}^2 2^{-\ell}\|\nablah v^\h_\mu\|_\twoB
\|v_F\|_{\Bh(t)}\|\wg\|_{\cB^{0,\f12}(t)}.
\end{equation}
Whereas by applying  \eqref{bhcb1} with
$a=v^\h_\mu,~b=v_F$ and $c=\nablah\wg,$  we obtain
$$
\Bigl|\int_0^t\bigl(\dvl(\vhg\otimes v_F)
\big|\dvl\nablah\wg\bigr)_{L^2}\,dt'\Bigr|
\lesssim d_{\ell}^2 2^{-\ell}\|\vhg\|_{\wt L^4_t(\cB_4^{0,\f12})}
\|v_F\|_{\Bh(t)}\|\nablah\wg\|_\twoB.$$
Inserting the above two estimates
into \eqref{5.21} and using \eqref{S2eq5a}, we achieve
\begin{equation}\begin{split}\label{5.23}
|{\rm II}_3|\lesssim d_{\ell}^2 2^{-\ell}\|v_F\|_{\Bh(t)}
\|\wg\|_{\cB^{0,\f12}(t)}\|\vhg\|_{\cB^{0,\f12}(t)}.
\end{split}\end{equation}

\noindent$\bullet$\underline{
The estimate of ${\rm II}_4$}

Due to $\diveh\baruh=0,$  by using  integration by parts, we write
$${\rm II}_4=\int_0^t\bigl(\dvl\diveh(\baruh v_F)
\,\big|\,\dvl\wg\bigr)_{L^2}\bfg(t')\,dt'
=-\int_0^t\bigl(\dvl(\baruh v_F)
\,\big|\,\dvl\nablah\wg\bigr)_{L^2} \bfg(t')\,dt'.$$
By applying Bony's decomposition \eqref{bony}, we get
$${\rm II}_4
=-\int_0^t\bigl(\dvl(T^\v_{\baruh} v_F+R^\v(\baruh, v_F))
\,\big|\,\dvl\nablah\wg\bigr)_{L^2}\bfg(t')\,dt'.$$
We first observe that
\begin{align*}
\int_0^t\bigl|&\bigl(\dvl(R^\v(\baruh, v_F))
\,\big|\,\dvl\nablah\wg\bigr)_{L^2}\bigr|\bfg(t')\,dt'\\
\lesssim &\sum_{\ell'\geq\ell-N_0}\int_0^t\bfg(t')\|\D_{\ell'}^\v\baruh(t')\|_{L^4_\h(L^2_\v)}\|S^\v_{\ell+2}v_F(t')\|_{L^4_\h(L^\infty_\v)}
\|\D_\ell^\v\na_\h w_\mu(t')\|_{L^2}\,dt'\\
\lesssim &\sum_{\ell'\geq\ell-N_0}2^{-\f{\ell'}2}\int_0^t d_{\ell'}(t')\bfg(t')\|\baruh(t')\|_{\cB^{0,\f12}_4}\|v_F(t')\|_{\cB^{0,\f12}_4}
\|\D_\ell^\v\na_\h w_\mu(t')\|_{L^2}\,dt'\\
\lesssim &\sum_{\ell'\geq\ell-N_0}d_{\ell'} 2^{-\f{\ell'}2}\int_0^t\bfg(t') \|\baruh(t')\|_{\cB^{0,\f12}_4}\|v_F(t')\|_{\cB^{0,\f12}_4}
\|\D_\ell^\v\na_\h w_\mu(t')\|_{L^2}\,dt',
\end{align*}
applying H\"older's inequality  and Proposition \ref{lemB4} gives
\begin{align*}
\int_0^t\bigl|&\bigl(\dvl(R^\v(\baruh, v_F))
\,\big|\,\dvl\nablah\wg\bigr)_{L^2}\bigr|\bfg(t')\,dt'\\
\lesssim &\sum_{\ell'\geq\ell-N_0}d_{\ell'} 2^{-\f{\ell'}2}\Bigl(\int_0^t\bfg^4(t') \|\baruh(t')\|_{\cB^{0,\f12}_4}^4\,dt'\Bigr)^{\f14}\|v_F\|_{L^4_t(\cB^{0,\f12}_4)}
\|\D_\ell^\v\na_\h w_\mu\|_{L^2_t(L^2)}\\
\lesssim& \mu^{-\f14}d_{\ell}^2 2^{-\ell}\|v_F\|_{\cB^{-\f12,\f12}_4(t)}\|\na_\h w_\mu\|_{L^2_t(\cB^{0,\f12})}.
\end{align*}
Along the same line, we find
\begin{align*}
\int_0^t\bigl|&\bigl(\dvl(T^\v_{\baruh}v_F)
\,\big|\,\dvl\nablah\wg\bigr)_{L^2}\bigr|\bfg(t')\,dt'\\
\lesssim &\sum_{|\ell'-\ell|\leq 5}\int_0^t\bfg(t')\|S_{\ell'-1}^\v\baruh(t')\|_{L^4_\h(L^\infty_\v)}\|\D^\v_{\ell}v_F(t')\|_{L^4_\h(L^2_\v)}
\|\D_\ell^\v\na_\h w_\mu(t')\|_{L^2}\,dt'\\
\lesssim &\sum_{|\ell'-\ell|\leq 5}\int_0^t\bfg(t') \|\baruh(t')\|_{\cB^{0,\f12}_4}\|\D_{\ell'}^\v v_F(t')\|_{L^4_\h(L^2_\v)}
\|\D_\ell^\v\na_\h w_\mu(t')\|_{L^2}\,dt'\\
\lesssim &\sum_{|\ell'-\ell|\leq 5}\Bigl(\int_0^t\bfg^4(t') \|\baruh(t')\|_{\cB^{0,\f12}_4}^4\,dt'\Bigr)^{\f14}\|\D_{\ell'}^\v v_F\|_{L^4_t(L^4_\h(L^2_\v))}
\|\D_\ell^\v\na_\h w_\mu\|_{L^2_t(L^2)}\\
\lesssim& \mu^{-\f14}d_{\ell}^2 2^{-\ell}\|v_F\|_{\wt{L}^4_t(\cB^{0,\f12}_4)}\|\na_\h w_\mu\|_{L^2_t(\cB^{0,\f12})}.
\end{align*}

As a result, it comes out
\begin{equation}\label{5.24}
|{\rm II}_4|\lesssim \mu^{-\f14}d_{\ell}^2 2^{-\ell}
\|v_F\|_{\Bh(t)}\|\nablah \wg\|_\twoB.
\end{equation}

\noindent$\bullet$\underline{
The estimates  of ${\rm II}_5$}

Due to $\pa_3v^3=-\diveh v^\h$ and $v^3=w+v_F,$ we write
\begin{align*}
{\rm II}_5&=\int_0^t\bigl(\dvl(-v^3\diveh\vhg)
\,\big|\,\dvl\wg\bigr)_{L^2}\,dt'\\
&=-\int_0^t\bigl(\dvl(v_F\diveh\vhg+\wg\diveh v^\h)
\,\big|\,\dvl\wg\bigr)_{L^2}.
\end{align*}
Then applying  \eqref{bhcb1p} gives rise to
\begin{equation}\begin{split}\label{5.25}
|{\rm II}_5| &\lesssim d_\ell^2 2^{-\ell}\|\nablah\vh\|_\twoB\bigl(\|v_F\|_{\Bh(t)}+\|w_\mu\|_{\cB^{0,\f12}(t)}\bigr)
\|\wg\|_{\cB^{0,\f12}(t)}\\
&\lesssim d_{\ell}^2 2^{-\ell}\|\vh\|_{\cB^{0,\f12}(t)}
\bigl(\|v_F\|_{\Bh(t)}
+\|\wg\|_{\cB^{0,\f12}(t)}\bigr)\|\wg\|_{\cB^{0,\f12}(t)}
.
\end{split}\end{equation}

\noindent$\bullet$\underline{
The estimates  of ${\rm II}_6$}

The estimate of  ${\rm II}_6$ can be handled
similarly as ${\rm I}_6$. Indeed in view of \eqref{5.12a}, we write
\beno
\begin{split}
\pa_3p=\pa_3(-\D)^{-1}\diveh\bigl(&v^\h\cdot\na_\h v^\h+\baruh\cdot\na_\h v^\h+v^\h\cdot\na_\h\baruh+\baruh\cdot\na_\h\baruh\\
&+v^3\pa_3\baruh+v^3\pa_3v^\h\bigr)+
\pa_3^2(-\D)^{-1}\bigl(v\cdot\na v^3+\baruh\cdot\na_\h v^3\bigr).
\end{split}
\eeno
Accordingly,  we have the decomposition
${\rm II}_6=\sum_{i=1}^5{\rm II}_{6,i}$ with
\begin{align*}
{\rm II}_{6,1}=&\int_0^t\bigl(\dvl\pa_3(-\D)^{-1}\diveh
\bigl(v^\h\cdot\nabla_\h\vhg+\baruh\cdot\nablah\vhg
+\vh_\mu\cdot\nablah\baruh\bigr)\,\big|\,\dvl\wg\bigr)_{L^2}\,dt',\\
{\rm II}_{6,2}=&\int_0^t\bigl(\dvl\pa_3(-\D)^{-1}\diveh
(v^3\pa_3\vhg)
\,\big|\,\dvl\wg\bigr)_{L^2}\,dt',\\
{\rm II}_{6,3}=&\int_0^t\bigl(\dvl\pa_3(-\D)^{-1}\diveh
(v^3_\mu\pa_3\baruh)\,\big|\,\dvl\wg\bigr)_{L^2}\,dt',\\
{\rm II}_{6,4}=&\int_0^t\bigl(\dvl\pa_3^2(-\D)^{-1}
(v\cdot\nabla v^3_\mu+\baruh\cdot\nablah v^3_\mu
)\,\big|\,\dvl\wg\bigr)_{L^2}\,dt',\\
{\rm II}_{6,5}=&\sum_{i=1}^2\sum_{j=1}^2\int_0^t
\bigl(2\dvl(-\D)^{-1}\pa_i\pa_j(\baru^i \pa_3\baru^j_\mu)
\,\big|\,\dvl\wg\bigr)_{L^2}\,dt'.
\end{align*}

It is easy to observe from the estimate of ${\rm I}_{6,1}$ that
\begin{equation}\begin{split}\label{5.26g}
|{\rm II}_{6,1}|\lesssim d_{\ell}^2 2^{-\ell}
\Bigl(&\|\vh\|_{\infB}^{\f12}\|\nablah\vh\|_\twoB^{\f32}\|\wg\|_{\infB}^{\f12}\|\nablah\wg\|_\twoB^{\f12}\\
&\qquad\qquad+\|\wg\|_\twogB^{\f12}\|\nablah\wg\|_\twoB^{\f12}\|\nablah\vhg\|_\twoB\Bigr).
\end{split}\end{equation}

While by using $\pa_3v^3=-\diveh v^\h$ and integration by parts, we write
\beno
\begin{split}
{\rm II}_{6,2}=&\int_0^t\bigl(\dvl\pa_3(-\D)^{-1}\diveh
[\pa_3(v^3\vhg)-\vhg\pa_3v^3]
\,\big|\,\dvl\wg\bigr)_{L^2}\,dt'\\
=&-\int_0^t\bigl(\dvl(-\D)^{-1}\pa_3^2
(v^3\vhg)
\,\big|\,\dvl\nablah w_\mu\bigr)_{L^2}\,dt'\\
&+\int_0^t\bigl(\dvl\pa_3(-\D)^{-1}\diveh
(\vh_\mu\diveh\vh)
\,\big|\,\dvl\wg\bigr)_{L^2}\,dt'
\eqdef {\rm II}_{6,2}^a+{\rm II}_{6,2}^b.
\end{split}
\eeno
It follows from \eqref{bhcb1} and $v^3=v_F+w$ that
\begin{align*}
\bigl|{\rm II}_{6,2}^a\bigr|
&\lesssim d_{\ell}^2 2^{-\ell}\|\vh\|_{\cB^{0,\f12}(t)}
\bigl(\|v_F\|_{\Bh(t)}
+\|\wg\|_{\cB^{0,\f12}(t)}\bigr)\|\wg\|_{\cB^{0,\f12}(t)}.
\end{align*}
Whereas by using a modified version of \eqref{fone3}, we infer
\beno
\bigl|{\rm II}_{6,2}^b\bigr|\lesssim d_{\ell}^2 2^{-\ell}
\|\vh\|_{\infB}^{\f12}\|\nablah\vh\|_\twoB^{\f32}
\|\wg\|_{\infB}^{\f12}\|\nablah\wg\|_\twoB^{\f12}.
\eeno
Therefore, we obtain
\begin{equation}\begin{split}\label{5.26c}
|{\rm II}_{6,2}|\lesssim d_{\ell}^2 2^{-\ell}
\|\vh\|_{\cB^{0,\f12}(t)}
\bigl(\|v_F\|_{\Bh(t)}+\|\vh\|_{\cB^{0,\f12}(t)}
+\|\wg\|_{\cB^{0,\f12}(t)}\bigr)\|\wg\|_{\cB^{0,\f12}(t)}.
\end{split}\end{equation}

Whereas applying \eqref{bhcb1p} with $a=\pa_3\baruh,~b=v^3_\mu$
and $c=\wg$ leads to
\begin{equation}\label{5.26b}
|{\rm II}_{6,3}|\lesssim d_{\ell}^2 2^{-\ell}\|\pa_3\baruh\|_\twoB
\bigl(\|v_F\|_{\Bh(t)}+\|\wg\|_{\cB^{0,\f12}(t)}\bigr)\|\wg\|_{\cB^{0,\f12}(t)}.
\end{equation}

On the other hand, again due to $\dive v=0,$ we write
\begin{align*}
{\rm II}_{6,4}=\int_0^t\bigl(\dvl\pa_3^2(-\D)^{-1}
\bigl(&v^\h\cdot\nabla_\h w_\mu+v^\h_\mu\cdot\nabla_\h v_F+v^3\pa_3v^3_\mu\\
&\qquad\quad+\baruh\cdot\nablah w_\mu+\baruh\cdot\nablah v_F
\bigr)\,\big|\,\dvl\wg\bigr)_{L^2}\,dt'.
\end{align*}
Noticing that $(-\D)^{-1}\pa_3^2$ is a bounded Fourier operator,
we observe that
${\rm II}_{6,4}$ shares the same estimate as
$\sum_{i=1}^5{\rm II}_{i}$ given before, that is
\begin{equation}\begin{split}\label{5.27}
|{\rm II}_{6,4}|\lesssim d_{\ell}^2 2^{-\ell}
\Bigl(&\|\vh\|_{\cB^{0,\f12}(t)}
\|\wg\|_{\cB^{0,\f12}(t)}^2
+\|\wg\|_\twogB^{\f12}\|\wg\|_{\cB^{0,\f12}(t)}^{\f32}\\
&\qquad+\|v_F\|_{\Bh(t)}\bigl(\mu^{-\f14}+\|\vhg\|_{\cB^{0,\f12}(t)}\bigr)\|\wg\|_{\cB^{0,\f12}(t)}\Bigr).
\end{split}\end{equation}

Finally since $(-\D)^{-1}\pa_i\pa_j$
is a bounded Fourier operator, we get, by applying \eqref{fone1}
with $a=\baruh,~b=\pa_3\baruh_\mu,~c=\wg$, that
\begin{equation}\label{5.28}
|{\rm II}_{6,5}|\lesssim d_{\ell}^2 2^{-\ell}\|\wg\|_\twogB^{\f12}
\|\nablah\wg\|_\twoB^{\f12}\|\pa_3\baruh_\mu\|_\twoB.
\end{equation}

By summing (\ref{5.26g}-\ref{5.28}), we arrive at
\begin{equation}\label{5.28g}
\begin{split}
|{\rm II}_{6}|\lesssim & d_{\ell}^2 2^{-\ell}\Bigl(\bigl(\|\vh\|_{\cB^{0,\f12}(t)}
+\|\pa_3\baruh\|_\twoB+\|\wg\|_{\cB^{0,\f12}(t)}\bigr)\|\wg\|_\twogB^{\f12}
\|\wg\|_{\cB^{0,\f12}(t)}^{\f12}\\
&+\bigl(\mu^{-\f14}+\|\pa_3\baruh_\mu\|_\twoB+\|v^\h\|_{\cB^{0,\f12}(t)}\bigr)\|v_F\|_{\cB^{-\f12,\f12}_4}\|w_\mu\|_{\cB^{0,\f12}(t)}\\
&+\bigl(\|\vh\|_{\cB^{0,\f12}(t)}+\|\pa_3\baruh\|_\twoB\bigr)\|w_\mu\|_{\cB^{0,\f12}(t)}^2+\|v^\h\|_{\cB^{0,\f12}(t)}^2
\|w_\mu\|_{\cB^{0,\f12}(t)}\Bigr).
\end{split}
\end{equation}

Let us now complete the proof of \eqref{aprioriv3}.

\begin{proof}[Proof of \eqref{aprioriv3}]
By inserting the estimates \eqref{5.19},~\eqref{5.20},
(\ref{5.23}-\ref{5.25}) and \eqref{5.28g} into \eqref{5.18},
and then
multiplying $2^{\ell+1}$ to the resulting inequality, and finally
taking square root and then summing up  the resulting inequalities over $\Z,$ we obtain
\begin{equation*}\begin{split}
\|\wg&\|_{\cB^{0,\f12}(t)}+\sqrt{2\mu}\|\wg\|_{\twogB}\\
\leq &\|u^3_{0,{\rm l}\h}\|_{\cB^{0,\f12}}+C\bigl(\|\vh\|_{\cB^{0,\f12}(t)}^{\f12}
+\|\pa_3\baruh_\mu\|_\twoB^{\f12}\bigr)\|w_\mu\|_{\cB^{0,\f12}(t)}\\
&+C\Bigl(\|v^\h\|_{\cB^{0,\f12}(t)}+\bigl(\mu^{-\f18}+\|\vh\|_{\cB^{0,\f12}(t)}^{\f12}+\|\pa_3\baruh\|_\twoB^{\f12}\bigr)
\|v_F\|_{\cB^{-\f12,\f12}_4}^{\f12}\Bigr)\|w_\mu\|_{\cB^{0,\f12}(t)}^{\f12}\\
&+C\bigl(\|\vh\|_{\cB^{0,\f12}(t)}^{\f12}
+\|\pa_3\baruh\|_\twoB^{\f12}+\|w_\mu\|_{\cB^{0,\f12}(t)}^{\f12}\bigr)\|\wg\|_\twogB^{\f14}
\|\nablah\wg\|_{\cB^{0,\f12}(t)}^{\f14}.
\end{split}\end{equation*}
Applying Young's inequality gives
\begin{align*}
C\Bigl(&\|v^\h\|_{\cB^{0,\f12}(t)}+\bigl(\mu^{-\f18}+\|\vh\|_{\cB^{0,\f12}(t)}^{\f12}+\|\pa_3\baruh\|_\twoB^{\f12}\bigr)
\|v_F\|_{\cB^{-\f12,\f12}_4}^{\f12}\Bigr)\|w_\mu\|_{\cB^{0,\f12}(t)}^{\f12}\\
\leq &\f1{12}\|w_\mu\|_{\cB^{0,\f12}(t)}+C\bigl(\mu^{-\f14}+\|\vh\|_{\cB^{0,\f12}(t)}+\|\pa_3\baruh\|_\twoB\bigr)\|v_F\|_{\cB^{-\f12,\f12}_4(t)}
+C\|v^\h\|_{\cB^{0,\f12}(t)}^2,
\end{align*}
and
\begin{align*}
C\bigl(\|\vh\|_{\cB^{0,\f12}(t)}^{\f12}&+\|\pa_3\baruh\|_\twoB^{\f12}+\|w_\mu\|_{\cB^{0,\f12}(t)}^{\f12}\bigr)\|\wg\|_\twogB^{\f14}
\|\wg\|_{\cB^{0,\f12}(t)}^{\f14}\\
\leq&\f1{12}\|w_\mu\|_{\cB^{0,\f12}(t)}+C\|\wg\|_\twogB+C\bigl(\|\vh\|_{\cB^{0,\f12}(t)}+\|\pa_3\baruh\|_\twoB\bigr).
\end{align*}
As a result, it comes out
\begin{align*}
\|\wg&\|_{\cB^{0,\f12}(t)}+\sqrt{2\mu}\|\wg\|_{\twogB}
\leq \|u^3_{0,{\rm l}\h}\|_{\cB^{0,\f12}}+C\|\wg\|_\twogB\\
&+\Bigl(\f16+C\bigl(\|\vh\|_{\cB^{0,\f12}(t)}^{\f12}
+\|\pa_3\baruh\|_\twoB^{\f12}\bigr)\Bigr)\|\wg\|_{\cB^{0,\f12}(t)}+C\Bigl(\|\vh\|_{\cB^{0,\f12}(t)}
\\
+&\|\pa_3\baruh\|_\twoB+\|\vh\|_{\cB^{0,\f12}(t)}^2+\bigl(\mu^{-\f14}+
\|\vh\|_{\cB^{0,\f12}(t)}+\|\pa_3\baruh\|_\twoB\bigr)\|v_F\|_{\cB^{-\f12,\f12}_4(t)}\Bigr).
\end{align*}
Taking $\mu$ in the above inequality so that $\sqrt{2\mu}=C$ gives rise to
\begin{equation}\begin{split}\label{5.29}
\Bigl(\f56&-C\bigl(\|\vh\|_{\cB^{0,\f12}(t)}^{\f12}
+\|\pa_3\baruh\|_\twoB^{\f12}\bigr)\Bigr)\|\wg\|_{\cB^{0,\f12}(t)}
\\
\leq &\|u^3_{0,{\rm l}\h}\|_{\cB^{0,\f12}}+C\Bigl(\|\vh\|_{\cB^{0,\f12}(t)}+\|\pa_3\baruh\|_\twoB
+\|\vh\|_{\cB^{0,\f12}(t)}^2\\
&\qquad\qquad\qquad+\bigl(1+
\|\vh\|_{\cB^{0,\f12}(t)}+\|\pa_3\baruh\|_\twoB\bigr)\|v_F\|_{\cB^{-\f12,\f12}_4(t)}\Bigr).
\end{split}\end{equation}
On the other hand,
in view of the definition of $u^3_{0,{\rm l}\h}$, there holds for any
$\ell\in\Z$ that
$$\|\dvl u^3_{0,{\rm l}\h}\|_{L^2}
\lesssim\sum_{|j-\ell|\leq1}\|S^\h_{j-1}\dvj u^3_0\|_{L^2}
\lesssim d_\ell 2^{-\f\ell2}\|u^3_0\|_{\Bh},$$
which indicates
$$\|u^3_{0,{\rm l}\h}\|_{\cB^{0,\f12}}\lesssim\|u^3_0\|_{\Bh}.$$
Inserting the above estimate into \eqref{5.29} and repeating the argument from \eqref{4.16} to \eqref{4.17},
we conclude the proof of  \eqref{aprioriv3}. \end{proof}

\medskip

\appendix

\setcounter{equation}{0}
\section{The proof of Lemmas \ref{lemfone} and \ref{lembhcb}}\label{appenda}

In this section, we present the proof of Lemmas \ref{lemfone} and \ref{lembhcb}.

\begin{proof}[Proof of Lemma \ref{lemfone}]
By applying Bony's decomposition in the  vertical variable \eqref{bony}
to $a\otimes b$, we write
\begin{equation}\begin{split}\label{fone4}
&\qquad\qquad\quad\int_0^T\bigl(\dvl A(D)(a\otimes b)
\big|\dvl c\bigr)_{L^2}\,dt=Q_1+Q_2\quad\mbox{with}\\
&Q_1\eqdef\int_0^T\bigl(\dvl A(D)(T^\v_{a} b)
\big|\dvl c\bigr)_{L^2}\,dt
=\int_0^T\bigl(\dvl(T^\v_{a} b)
\big| A(D)\dvl  c\bigr)_{L^2}\,dt\quad \mbox{and}\\
&Q_2\eqdef\int_0^T\bigl(\dvl A(D) R^\v(a, b)
\big|\dvl c\bigr)_{L^2}\,dt
=\int_0^T\bigl(\dvl R^\v(a, b)
\big| A(D)\dvl c\bigr)_{L^2}\,dt.
\end{split}\end{equation}
Considering the support properties to the Fourier transform of the terms in $T^\v_{a} b$, and noting that $A(D)$ is a smooth
homogeneous Fourier multiplier of degree zero, we find
\begin{equation*}\begin{split}\label{fone5}
|Q_1|
&\leq\int_0^T\|\dvl(T^\v_{a} b)\|_{L_\h^{\f43}(L_\v^2)}
\| A(D) \dvl c\|_{L_\h^4(L_\v^2)}\,dt\\
&\lesssim \sum_{|\ell'-\ell|\leq 5} \int_0^T\|S^\v_{\ell'-1}a\|_{L_\h^4(L_\v^\infty)}
\|\dvlp b\|_{L^2}\|A(D) \dvl  c\|_{L^2}^{\f12}
\|\nablah A(D)\dvl  c\|_{L^2}^{\f12}\,dt\\
&\lesssim\sum_{|\ell'-\ell|\leq 5}
\Bigl(\int_0^T\|S^\v_{\ell'-1}a(t)\|_{L_\h^4(L_\v^\infty)}^4
\|\dvl c(t)\|_{L^2}^2\,dt\Bigr)^{\f14}\|\dvlp b\|_{L^2_T(L^2)}\|\nablah\dvl c\|_{L^2_T(L^2)}^{\f12}.
\end{split}\end{equation*}
It follows from Lemma \ref{lemBern} and Definition \ref{Def2.4} that
\begin{equation*}\label{fone6}
\|S^\v_{\ell'-1}a(t)\|_{L_\h^4(L_\v^\infty)}
\leq\sum_{j\leq\ell'-2}\|\D^\v_{j}a(t)\|_{L_\h^4(L_\v^\infty)}
\lesssim\sum_{j\leq\ell'-2}2^{\f{j}2}
\|\D^\v_{j}a(t)\|_{L_\h^4(L_\v^2)}\lesssim \|a(t)\|_{\cB_4^{0,\f12}}.
\end{equation*}
This together with Definition \ref{defpz} ensures that
\begin{equation}\label{fone7}
|Q_1|
\lesssim d_{\ell}^2 2^{-\ell}\|c\|_{\wt{L}^2_{T,\ff}(\cB^{0,\f12})}^{\f12} \|b\|_\twoBT
\|\nablah c\|_\twoBT^{\f12}.
\end{equation}
Along the same line, we get, by applying \eqref{S2eq5}, that
\begin{equation}\begin{split}\label{fone8}
|Q_{1,\fg}|&\eqdef\int_0^T\bigl|\bigl(\dvl(T^\v_{a} b)
\big| A(D)\dvl  c\bigr)_{L^2}\bigr|\fg^2\,dt\\
&\lesssim\sum_{|\ell'-\ell|\leq 5}
\|\sqrt{\fg} S^\v_{\ell'-1}a\|_{L_T^4(L_\h^4(L_\v^\infty))}\|\fg\dvlp b\|_{L^2_T(L^2)}
\|\dvl c\|_{L^\infty_T(L^2)}^{\f12}\|\fg\nablah\dvl c\|_{L^2_T(L^2)}^{\f12}
\\
&\lesssim d_{\ell}^2 2^{-\ell}\|a\|_\infBT^{\f12}\|\fg\nablah a\|_\twoBT^{\f12}
\|\fg b\|_\twoBT\|c\|_\infB^{\f12}\|\fg\nablah c\|_\twoBT^{\f12}.
\end{split}\end{equation}

On the other hand, once again
considering the support properties to the Fourier transform of the terms in $R^\v({a}, b),$ we find
\begin{equation*}\begin{split}\label{fone9}
|Q_2|
&\leq\int_0^T\|\dvl R^\v(a,b)\|_{L_\h^{\f43}(L_\v^2)}
\| A(D)\dvl c\|_{L_\h^4(L_\v^2)}\,dt\\
&\lesssim\sum_{\ell'\geq\ell-N_0}\int_0^T\|\D^\v_{\ell'}a\|_{L_\h^4(L_\v^2)}
\|S^\v_{\ell'+2}b\|_{L^2_\h(L^\infty_\v)}
\| A(D)\dvl c\|_{L^2}^{\f12}
\|\nablah A(D)\dvl c\|_{L^2}^{\f12}\,dt\\
&\lesssim\sum_{\ell'\geq\ell-N_0} 2^{-\f{\ell'}2}\int_0^Td_{\ell'}(t)\|a(t)\|_{\cB^{0,\f12}_4}\|b(t)\|_{L^2_\h(L^\infty_\v)}
\|\dvl c(t)\|_{L^2}^{\f12}\|\nablah\dvl c(t)\|_{L^2}^{\f12}\,dt\\
&\lesssim\sum_{\ell'\geq\ell-N_0} d_{\ell'}2^{-\f{\ell'}2}\int_0^T\|a(t)\|_{\cB^{0,\f12}_4}\|b(t)\|_{L^2_\h(L^\infty_\v)}
\|\dvl c(t)\|_{L^2}^{\f12}\|\nablah\dvl c(t)\|_{L^2}^{\f12}\,dt.
\end{split}\end{equation*}
Yet it follows from Lemma \ref{lemBern} that
$$\|b\|_{L^2_T(L^2_\h(L^\infty_\v))}
\lesssim\sum_{\ell\in\Z}2^{\f{\ell}2}
\|\D^\v_{\ell}b\|_{L_T^2(L^2)}\leq\|b\|_\twoBT.$$
As a result, by virtue of Definition \ref{defpz}, we obtain
\begin{equation}\begin{split}\label{fone10}
|Q_2|
&\lesssim\sum_{\ell'\geq\ell-N_0}d_{\ell'} 2^{-\f{\ell'}2}
\Bigl(\int_0^T\|a(t)\|_{\cB_4^{0,\f12}}^4
\|\dvl c(t)\|_{L^2}^{2}\,dt\Bigr)^{\f14}
\|\nablah\dvl c\|_{L_T^2(L^2)}^{\f12}\|b\|_\twoBT\\
&\lesssim\sum_{\ell'\geq\ell-N_0}d_{\ell'} 2^{-\f{\ell'}2}
\Bigl(d_{\ell} 2^{-\f{\ell}2}\|c\|_\twofBT\Bigr)^{\f12}
\bigl(d_{\ell} 2^{-\f{\ell}2}\|\nablah c\|_\twoBT\bigr)^{\f12}
\|b\|_\twoBT\\
&\lesssim d_{\ell}^2 2^{-\ell}
\|c\|_\twofBT^{\f12}\|\nablah c\|_\twoBT^{\f12} \|b\|_\twoBT.
\end{split}\end{equation}
Similarly,  thanks to \eqref{S2eq5}, one has
\begin{equation}\begin{split}\label{fone11}
|&Q_{2,\fg}|\eqdef\int_0^T\bigl|\bigl(\dvl R^\v(a, b)
\big|  A(D)\dvl c\bigr)_{L^2}\bigr|\fg^2\,dt\\
&\lesssim\sum_{\ell'\geq\ell-N_0}\|\sqrt{\fg}\D^\v_{\ell'}a\|_{L^4_T(L_\h^4(L_\v^2))}
\|\fg S^\v_{\ell'+2}b\|_{L^2_T(L^2_\h(L^\infty_\v))}
\bigl(\|\dvl c\|_{L^2_T(L^2)}\|\fg\nablah\dvl c\|_{L^2_T(L^2)}\bigr)^{\f12}\\
&\lesssim d_{\ell}^2 2^{-\ell}\|a\|_\infBT^{\f12}\|\fg\nablah a\|_\twoBT^{\f12}
\|\fg b\|_\twoBT\|c\|_\infBT^{\f12}\|\fg\nablah c\|_\twoBT^{\f12}.
\end{split}\end{equation}
Combining \eqref{fone7} with \eqref{fone10} gives \eqref{fone1}.
And \eqref{fone3}
follows from \eqref{fone8} and \eqref{fone11}.

It remains to prove \eqref{fone2}. Indeed similar to the proof of \eqref{fone7}, we  write
\begin{align*}
|Q_1|
&\lesssim \sum_{|\ell'-\ell|\leq 5} \int_0^T
\|S^\v_{\ell'-1}a\|_{L_\h^4(L_\v^\infty)}
\|\dvlp b\|_{L_\h^4(L_\v^2)}\| A(D)\dvl c\|_{L^2}\,dt\\
&\lesssim \sum_{|\ell'-\ell|\leq 5} \int_0^T
\|a(t)\|_{\cB^{0,\f12}_4}
\|\dvlp b(t)\|_{L^2}^{\f12}\|\dvlp\nablah b(t)\|_{L^2}^{\f12}\|\dvl c(t)\|_{L^2}\,dt\\
&\lesssim \sum_{|\ell'-\ell|\leq 5} \Bigl(\int_0^T
\|a(t)\|_{\cB^{0,\f12}_4}^4
\|\dvlp b(t)\|_{L^2}^2\,dt\Bigr)^{\f14}\|\dvlp\nablah b\|_{L^2_T(L^2)}^{\f12}\|\dvl c\|_{L^2_T(L^2)},
\end{align*}
from which and Definition \ref{defpz}, we infer
\begin{equation}\begin{split}\label{fone12}
|Q_1|
&\lesssim d_\ell 2^{-\f{\ell}2} \sum_{|\ell'-\ell|\leq 5}
d_{\ell'} 2^{-\f{\ell'}2}\|b\|_\twofBT^{\f12}
\|\nablah b\|_\twoBT^{\f12}\|c\|_\twoBT\\
&\lesssim d_{\ell}^2 2^{-\ell}
\|b\|_\twofBT^{\f12}\|\nablah b\|_\twoBT^{\f12}\|c\|_\twoBT.
\end{split}\end{equation}
While we deduce from Definition \ref{Def2.4} that
\begin{align*}
|Q_2|
&\lesssim\sum_{\ell'\geq\ell-N_0}\int_0^T\|\D^\v_{\ell'}a\|_{L_\h^4(L_\v^2)}
\|S^\v_{\ell'+2}b\|_{L^4_\h(L^\infty_\v)}
\| A(D)\dvl c\|_{L^2}\,dt\\
&\lesssim\sum_{\ell'\geq\ell-N_0}d_{\ell'}2^{-\f{\ell'}2}\int_0^T\|a(t)\|_{\cB^{0,\f12}_4}
\|b(t)\|_{L^4_\h(L^\infty_\v)}
\|\dvl c(t)\|_{L^2}\,dt\\
&\lesssim d_\ell 2^{-\f{\ell}2}\|c\|_\twoBT
\sum_{\ell'\geq\ell-4} d_{\ell'} 2^{-\f{\ell'}2}
\Bigl(\int_0^T \|a(t)\|_{\cB_4^{0,\f12}}^2
\|b(t)\|_{L^4_\h(L^\infty_\v)}^{2}\,dt\Bigr)^{\f12}.
\end{align*}
Whereas we get, by applying triangle inequality and Lemma \ref{lemBern}, that
\begin{align*}
\Bigl(\int_0^T&\|a(t)\|_{\cB_4^{0,\f12}}^2
\|b(t)\|_{L^4_\h(L^\infty_\v)}^{2}\,dt\Bigr)^{\f12}\\
\lesssim &\sum_{\ell\in\Z}2^{\f{\ell}2}\Bigl(\int_0^T \|a(t)\|_{\cB_4^{0,\f12}}^2
\|\D_\ell b(t)\|_{L^2}\|\na_\h\D_\ell b(t)\|_{L^2}\,dt\Bigr)^{\f12}\\
\lesssim &\sum_{\ell\in\Z}2^{\f{\ell}2}\Bigl(\int_0^T \|a(t)\|_{\cB_4^{0,\f12}}^4
\|\D_\ell b(t)\|_{L^2}^2\,dt\Bigr)^{\f14}\|\na_\h\D_\ell b\|_{L^2_T(L^2)}^{\f12}\\
\lesssim &\|b\|_\twofBT^{\f12} \|\nablah b\|_\twoBT^{\f12}.
\end{align*}
This in turn shows that
\begin{equation*}\begin{split}\label{fone13}
|Q_2|\lesssim  d_{\ell}^2 2^{-\ell}\| c\|_\twoBT
\|b\|_\twofBT^{\f12}\|\nablah b\|_\twoBT^{\f12},
\end{split}\end{equation*}
which together with \eqref{fone12}  ensures
 \eqref{fone2}. This completes the proof of Lemma \ref{lemfone}.
\end{proof}

\begin{proof}[Proof of Lemma \ref{lembhcb}] Let $Q_1$ be given by \eqref{fone4}. We first get, by a similar derivation of \eqref{fone7}, that
\begin{equation*}\begin{split}\label{bhcb2}
|Q_1|
&\lesssim\sum_{|\ell'-\ell|\leq 5}\|S^\v_{\ell'-1}a\|_{L^4_T(L_\h^4(L_\v^\infty))}
\|\dvlp b\|_{L^4_T(L^4_\h(L^2_\v))}
\|A(D)\dvl c\|_{L^2_t(L^2)}\\
&\lesssim d_\ell 2^{-\f{\ell}2}\sum_{|\ell'-\ell|\leq 5}d_{\ell'}2^{-\f{\ell'}2}\|a\|_{\wt L^4_T(\cB_4^{0,\f12})}
\|b\|_{\wt L^4_T(\cB_4^{0,\f12})}
\|c\|_{\wt L^2_T(\cB^{0,\f12})},
\end{split}\end{equation*} which together with Proposition \ref{lemB4} implies that
\begin{equation}\label{bhcb7}
|Q_1|\lesssim d_{\ell}^2 2^{-\ell}\|a\|_{\wt L^4_T(\cB_4^{0,\f12})}
\|b\|_{\Bh(T)}\|c\|_\twoBT.
\end{equation}

While for $Q_2$ given by \eqref{fone4},  we  get, by a similar derivation of \eqref{fone10}, that
\begin{equation*}\begin{split}\label{bhcb8}
|Q_2|
&\lesssim\sum_{\ell'\geq\ell-N_0}\|\dvlp a\|_{L^4_T(L_\h^4(L_\v^2))}
\|S^\v_{\ell'+2}b\|_{L^4_T(L^4_\h(L^\infty_\v))}
\| A(D)\dvl c\|_{L^2_T(L^2)}\\
&\lesssim d_\ell 2^{-\f{\ell}2}\sum_{\ell'\geq\ell-N_0}
d_{\ell'} 2^{-\f{\ell'}2}\|a\|_{\wt L^4_T(\cB_4^{0,\f12})}
\|b\|_{\wt L^4_T(\cB_4^{0,\f12})}
\|c\|_{\wt L^2_T(\cB^{0,\f12})},
\end{split}\end{equation*}
from which and Proposition \ref{lemB4}, we infer
$$
|Q_2|
\lesssim d_{\ell}^2 2^{-\ell}\|a\|_{\wt L^4_T(\cB_4^{0,\f12})}
\|b\|_{\Bh(T)}\|c\|_\twoBT.$$
This together  with \eqref{fone4} and  \eqref{bhcb7}
ensures \eqref{bhcb1}.

The inequality \eqref{bhcb1p} can be proved similarly. As a matter of fact, we observe that
\begin{align*}
|Q_1|
&\lesssim\sum_{|\ell'-\ell|\leq 5}\|S^\v_{\ell'-1}a\|_{L^2_T(L_\h^2(L_\v^\infty))}
\|\dvlp b\|_{L^4_T(L^4_\h(L^2_\v))}
\|A(D)\dvl c\|_{L^4_T(L^4_\h(L^2_\v))}\\
&\lesssim\sum_{|\ell'-\ell|\leq 5}\|S^\v_{\ell'-1}a\|_{L^2_T(L_\h^2(L_\v^\infty))}
\|\dvlp b\|_{L^4_T(L^4_\h(L^2_\v))}
\|\dvl c\|_{L^\infty_T(L^2)}^{\f12}
\|\dvl \nablah c\|_{L^2_T(L^2)}^{\f12}\\
&\lesssim d_\ell 2^{-\f{\ell}2}\sum_{|\ell'-\ell|\leq 5}d_{\ell'}2^{-\f{\ell'}2}\|a\|_{\wt L^2_T(\cB^{0,\f12})}
\|b\|_{\wt{L}^4_T(\cB_4^{0,\f12})}
\|c\|_{\cB^{0,\f12}(T)},
\end{align*}
and
\begin{align*}
|Q_2|
&\lesssim\sum_{\ell'\geq\ell-N_0}\|\dvlp a\|_{L^2_T(L^2)}
\|S^\v_{\ell'+2}b\|_{L^4_T(L^4_\h(L^\infty_\v))}
\|A(D) \dvl c\|_{L^4_T(L_\h^4(L_\v^2))}\\
&\lesssim d_\ell 2^{-\f{\ell}2}\sum_{\ell'\geq\ell-N_0}
d_{\ell'} 2^{-\f{\ell'}2}\|a\|_{\wt L^2_T(\cB^{0,\f12})}
\|b\|_{\wt{L}^4_T(\cB_4^{0,\f12})}
\|c\|_{\cB^{0,\f12}(T)}.
\end{align*}
Then  \eqref{bhcb1p} follows from Proposition \ref{lemB4}.
This completes the proof of this lemma.
\end{proof}

\bigbreak \noindent {\bf Acknowledgments.} M. Paicu was partially supported by the Agence Nationale de la
Recherche, Project IFSMACS, Grant ANR-15-CE40-0010.
P. Zhang is partially supported
by NSF of China under Grants   11371347 and 11688101, Morningside Center of Mathematics of The Chinese Academy of Sciences and innovation grant from National Center for
Mathematics and Interdisciplinary Sciences.

\end{document}